\documentclass{article}

\usepackage{graphicx,psfrag}
\usepackage{amsthm}
\usepackage{MnSymbol}
\usepackage{enumitem}

\newtheorem{thm}{Theorem}[section]
\newtheorem{cor}[thm]{Corollary}
\newtheorem{lem}[thm]{Lemma}

\newtheorem{conj}{Conjecture}
\theoremstyle{definition}
\newtheorem*{rem}{Remark}
\newtheorem*{exmpl}{Example}
\newtheorem{quest}{Question}

\newcommand{\tw}{{\rm{tw}}}
\newcommand{\stw}{{\rm{stw}}}
\newcommand{\dg}{{\rm{dg}}}
\newcommand{\pa}{{\rm{pa}}}

\title{Edge-intersection graphs of grid paths:\\the bend-number}
\author{Daniel Heldt, Kolja Knauer\thanks{Research was supported by the DFG as a GraDR EUROGIGA project.}, Torsten Ueckerdt\thanks{Research was supported by GraDR EUROGIGA project No. GIG/11/E023.}}

\begin{document}

\maketitle

\begin{abstract}
We investigate edge-intersection graphs of paths in the plane grid, regarding a parameter called the bend-number. I.e., every vertex is represented by a grid path and two vertices are adjacent if and only if the two grid paths share at least one grid-edge. The bend-number is the minimum~$k$ such that grid-paths with at most~$k$ bends each suffice to represent a given graph. This parameter is related to the interval-number and the track-number of a graph. We show that for every~$k$ there is a graph with bend-number~$k$. Moreover we provide new upper and lower bounds of the bend-number of graphs in terms of degeneracy, treewidth, edge clique covers and the maximum degree. Furthermore we give bounds on the bend-number of~$K_{m,n}$ and determine it exactly for some pairs of $m$ and $n$. Finally, we prove that recognizing single-bend graphs is NP-complete, providing the first such result in this field.
\end{abstract}

%%%%%%%%%%%%%%%%%%%%%%%%%%%%%%%%%%%%%%%%%%%%%%%%%%%%%%%%%%%%
%%%%%%%%%%%%%%%%%%%%%%%%%%%%%%%%%%%%%%%%%%%%%%%%%%%%%%%%%%%%
%%%%%%%%%%%%%%%%%%%%%%%%%%%%%%%%%%%%%%%%%%%%%%%%%%%%%%%%%%%%
\section{Introduction}\label{sec:int}
Golumbic, Lipshteyn and Stern~\cite{Gol-09} introduced \emph{edge-intersection graphs of paths on a grid} (EPG graphs), a concept arising from VLSI grid layout problems~\cite{Bra-90}. A simple graph $G$ is an EPG graph, if there is an assignment of paths in the plane grid to the vertices, such that two vertices are adjacent if and only if the corresponding paths intersect in at least one grid-edge. The assignment is then called an EPG representation of $G$. EPG graphs generalize \emph{edge-intersection graphs of paths on degree 4 trees} as considered by Golumbic, Lipshteyn and Stern in~\cite{Gol-08}. In~\cite{Gol-09} it is shown that every graph is an EPG graph, however a certain parameter of EPG representations has awoken some interest. The \emph{bend-number} $b(G)$ of $G$ is the minimum $k$, such that $G$ has an EPG representation, with each path having at most $k$ bends. Here a \emph{bend of a grid path} is a switch in its direction between horizontal and vertical. Figure~\ref{fig:k35} shows an EPG representation of $K_{3,10}$ where each path has at most two bends, hence $b(K_{3,10})\leq 2$. Generally, a graph $G$ with $b(G)\leq k$ is referred to as a \emph{$k$-bend graph}.

\begin{figure}[htb] 
 \centering
 \includegraphics{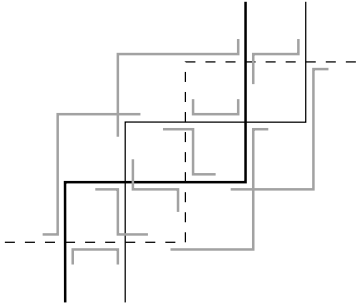}
 \caption{A 2-bend representation of $K_{3,10}$.}
 \label{fig:k35}
\end{figure}
% 
% \begin{figure}[htb] 
%  \psfrag{a}[cc][cc]{$a$}
%  \psfrag{b}[cc][cc]{$b$}
%  \psfrag{c}[cc][cc]{$c$}
%  \psfrag{d}[cc][cc]{$d$}
%  \psfrag{e}[cc][cc]{$e$}
%  \psfrag{f}[cc][cc]{$f$}
%  \centering
%  \includegraphics{pics/k33-1}
%  \caption{The $K_{3,3}\backslash{e}$ is a single-bend graph.}
%  \label{fig:k33-1}
% \end{figure}

\begin{rem}
 Most of the literature concerning this topic, including~\cite{Asi-09,Gol-09,Bie-10}, is considering $B_k$, the \emph{class of $k$-bend graphs}. Clearly $b(G) \leq k$ just paraphrases $G\in B_k$. However, we prefer to use $b(G)$ rather than $B_k$.
\end{rem}

\noindent Graphs with bend-number at most~1, called \emph{single-bend graphs}, already aroused interest in several respects, as seen in~\cite{Gol-09,Rie-09,Asi-12,Cam-12}. In~\cite{Asi-09,Bie-10} it has been shown that the bend-number of a graph can be arbitrarily large. Hence it is interesting to determine graphs or graph classes with bounded bend-number. Asinowski and Suk~\cite{Asi-09} give bounds on the bend-number of complete bipartite graphs. In~\cite{Hel-12} it is shown that $b(G)\leq 4$ for every planar $G$, $b(G)\leq 3$ for planar graphs with tree-width at most $3$ and that this is best-possible, and that $b(G)\leq 2$ for every $G$ with tree-width at most $2$, which includes outerplanar graphs and is best-possible due to an example of Biedl and Stern~\cite{Bie-10}. Biedl and Stern~\cite{Bie-10} also give upper bounds on $b(G)$ in terms of treewidth, pathwidth, degeneracy and maximum degree of $G$. 

\subsubsection*{Comparing parameters}\label{sec:Bs}

\emph{Interval graphs} are intersection graphs of intervals on the real line. Every vertex is associated with an interval, in such a way that two intervals overlap if and only if the corresponding vertices are adjacent. This subject has been extended to intersection graphs of systems of intervals in two ways:
\begin{itemize}
 \item In a \emph{$k$-interval representation} of a graph $G$ every vertex is associated with a set of at most $k$ intervals on the real line, such that vertices are adjacent iff any of their intervals intersect. The \emph{interval-number $i(G)$} is then defined as the minimum $k$, such that $G$ has a $k$-interval representation, see~\cite{Har-79}.
\item In a \emph{$k$-track representation} of a graph $G$ there are $k$ parallel lines, called \emph{tracks}. Every vertex is associated with one interval from each track. Again vertex adjacency is equivalent to interval intersection and the \emph{track-number $t(G)$} is the minimum $k$, such that $G$ has a $k$-track representation, see~\cite{Gya-95}. 
\end{itemize}
% 
% It is easy to see that $i(G) \leq t(G)$, since a $k$-track representation can be transferred into a $k$-interval representation by putting the tracks in any order on a single real line.

For emphasis we repeat:
\begin{itemize}
 \item A \emph{$k$-bend representation} is an EPG representation where each vertex is represented by a path with at most $k$ bends, i.e., at most $k+1$ segments. In this sense such a representation associates every vertex with at most $k+1$ intervals. The \emph{bend-number} $b(G)$ is the minimum $k$, such that $G$ has a $k$-bend representation.
\end{itemize}

Thus, the bend-number is yet another way to measure how far a graph is from being an interval-graph. Note that interval graphs are precisely the graphs with $i(G) = t(G) = b(G)+1 = 1$.

Now, $b(G)$ can be set in relation to $i(G)$ and $t(G)$: $b(G)$ is only a constant factor away from $i(G)$ and $t(G)$. Consecutive intervals representing a vertex in a $k$-interval representation may be connected by introducing three segments such that they form a grid-path. It is easy to see, that a $k$-track representation can be transformed into a $k$-interval representation, by putting the tracks on a single line, i.e., $i(G)\leq t(G)$. Thus one gets $b(G) \leq 4 (i(G)-1)\leq 4  (t(G)-1)$. On the other hand the grid-lines of a $k$-bend representation can be stringed together on a single line, i.e., $i(G) \leq b(G)+1$. We believe that no such bound exists for the track-number.

\begin{conj}\label{conj:tib}
%  Kevin Milans~\cite{Mil-11} proved that $t(L(K_n)) \in \Omega(\log^*(n))$, while $i(G) \leq 2$ for every line graph $G$. In particular, 
There is no function $f$ such that $t(G)\leq f(b(G))$ for all $G$ and equivalently no such function exists replacing $b$ by $i$. We suspect that the line graph of $K_n$ is a good candidate for showing this, i.e., that this family has unbounded track-number, whereas in~\cite{Bie-10} it is shown that it has bend-number at most $2$.
\end{conj}

\begin{table}[htb]
 \centering
 \begin{tabular}{l|cr|cr|cr}
  & \multicolumn{2}{c|}{$i(G)$} & \multicolumn{2}{c|}{$t(G)$} & \multicolumn{2}{c}{$b(G)$} \\
 \hline
forest & $2$ & \cite{Har-79} & $2$ & \cite{Deo-94} & $1$ & \cite{Gol-09} \\
outerplanar & $2$ & \cite{Sch-83} & $2$ & \cite{Kos-99} & $2$ & \cite{Hel-12} \\%Thm~\ref{thm:outerplanar} \\
planar & 3 & \cite{Sch-83} & 4 & \cite{Gon-07}  & $3\leq . \leq 4$ & \cite{Hel-12} \\
\hfill + bipartite & 3 & \cite{Sch-83} & 4 & \cite{Gon-07,Gon-09} & 3 & \cite{Bie-10} \\
line graph & 2 & &$?$ & Conj.~\ref{conj:tib} & $2$ & \cite{Bie-10} \\
$\dg(G)\leq k$ & $k+1$ & \cite{Kna-12} & $2k$ & \cite{Alo-92,Hak-96,Kna-12}  & $2k-1$ & Sec.~\ref{sec:acy} \\
$\tw(G)\leq k$ & $k+1$ & \cite{Din-98,Kna-12} & $k+1$ & \cite{Din-98,Kna-12} & $2k-2$ & Sec.~\ref{sec:tw} \\
degree $\leq\Delta$ & $\lceil \frac{\Delta+1}{2} \rceil$ & \cite{Gri-80} & $\leq \frac{3\Delta+6}{5}$ & \cite{Gul-86} & $\lceil \frac{\Delta}{2} \rceil\leq . \leq \Delta$ & Cor.~\ref{cor:chi'}\\
\end{tabular}
\caption{Some graph classes and their maximum interval-number, track-number and bend-number. Here $\dg(G)$ and $\tw(G)$ denotes the degeneracy and treewidth of $G$, respectively.}
\label{tab:results}
\end{table}

Many extremal questions about interval-numbers and track-numbers have been studied. In Table~\ref{tab:results} we have listed some considered graph classes and the maximum $i(G)$, $t(G)$ and $b(G)+1$ among all $G$ in this class. In the last two columns (corresponding to the track-number and the bend-number) some values remain unknown, yielding several interesting problems to attack.%In particular this includes the bend-number of graphs with maximum degree $\Delta$ (c.f. Section~\ref{sec:cli} for both) and planar graphs. 

\subsubsection*{Our Results}

\begin{itemize}
 \item In Section~\ref{sec:cli} we bound the bend-number of a graph in terms of its global and local clique covering number. This generalizes the bound for line graphs from~\cite{Bie-10} and improves it for line graphs of bipartite graphs. As a corollary we obtain that $b(G) \leq \Delta+1$ where $\Delta$ denotes the maximum degree of $G$, which improves the previous bound of $2\lceil \frac{\Delta+1}{2}\rceil+1$ from~\cite{Bie-10}.
 \item In Section~\ref{sec:bip} we present two lower bounds on the bend-number of the complete bipartite graph $K_{m,n}$. From the first we obtain that $b(K_{n,n}) = \lceil \frac{n}{2} \rceil$, which in particular proves that for every $k$ there is a $k$-bend graph that is not a $(k-1)$-bend graph. This confirms a conjecture of~\cite{Gol-09} and has been shown for even $k$ in~\cite{Bie-10}. With our second lower bound we improve the bound from~\cite{Bie-10} on the minimal $n$ for which $b(K_{m,n}) = 2m-2$. Moreover we show that this new bound is almost tight, disproving a conjecture of Biedl and Stern~\cite{Bie-10}.
 \item In Section~\ref{sec:acy} we prove that $b(G) \leq 2\dg(G)-1$ for all graphs $G$, where $\dg(G)$ denotes the degeneracy of $G$. This was suspected in~\cite{Bie-10}, where the authors prove a bound of $2\dg(G)+1$. We additionally show that our new bound is best-possible even for bipartite graphs.
 \item In Section~\ref{sec:tw} we prove that $b(G) \leq 2\tw(G)-2$ for all graphs $G$, where $\tw(G)$ denotes the treewidth of $G$, generalizing the known $b(K_{m,n}) \leq 2m-2$~\cite{Bie-10}. This improves a result of Biedl and Stern~\cite{Bie-10} who achieve the same bound, but with pathwidth instead of treewidth. Our bound is best-possible.
 \item In Section~\ref{sec:rec} we present the first hardness result in the field of EPG graphs. In particular we prove that recognizing single-bend graphs is NP-complete, answering a question that has been frequently asked~\cite{Gol-09,Bie-10}. The recognition of $k$-bend graphs for $k \geq 2$ remains open.
\end{itemize}

%%%%%%%%%%%%%%%%%%%%%%%%%%%%%%%%%%%%%%%%%%%%%%%%%%%%%%%%%%%%
%%%%%%%%%%%%%%%%%%%%%%%%%%%%%%%%%%%%%%%%%%%%%%%%%%%%%%%%%%%%
%%%%%%%%%%%%%%%%%%%%%%%%%%%%%%%%%%%%%%%%%%%%%%%%%%%%%%%%%%%%
\section{Preliminaries}\label{sec:pre}

We consider simple undirected graphs $G$ with vertex set $V(G)$ and edge set $E(G)$. An EPG representation is a set of finite paths $\{P(v) \mid v \in V(G)\}$, which consist of consecutive edges of the rectangular grid in the plane, such that $\{v,w\} \in E(G)$ if and only if $P(v) \cap P(w)$ contains a grid-edge. In particular two paths representing \emph{non}-adjacent vertices may intersect in grid-points. A bend of $P(v)$ is a point of $P(v)$, where a horizontal grid-edge of $P(v)$ is followed by a vertical grid-edge of $P(v)$.

The set of grid-edges between two consecutive bends or the first (last) bend and the start (end) of $P(v)$ is called a \emph{segment}. So a $k$-bend path consists of $k+1$ segments, each of which is either horizontal or vertical. A \emph{subsegment} is a connected subset of a segment. Two horizontal subsegments in an EPG representation \emph{see each other} if there is a vertical grid-line intersecting both subsegments. Similarly, two vertical subsegments see each other if there is a horizontal grid-line intersecting both. We say that a subsegment $s \subset P(v)$ \emph{displays $v$} if the grid-edges on $s$ are exclusively in $P(v)$ and in no other path. Similarly, a segment $s \subset P(v)\cap P(w)$ \emph{displays the edge $\{v,w\}$} if every grid-edge of $s$ is only contained in $P(u)\cap P(v)$ and not element of any other path.

Two special types of paths appear more frequently in the paper: A path $P$ is a \emph{staircase} if going along $P$ we see alternating left turns and right turns, and a path $P$ with an even number of bends is called a \emph{snake} if we see alternating two left turns and two right turns.

For a set $S$ of horizontal (respectively vertical) subsegments in an EPG representation that pairwise see each other we say that a grid path $P$ \emph{connects $S$} if $P$ is a snake and every horizontal (respectively vertical) segment of $P$ is contained in a different subsegment from $S$. More formally, there are two vertical (respectively horizontal) grid-lines $\ell_1$ and $\ell_2$ that intersect all subsegments in $S$ and contain all vertical (respectively horizontal) segments of $P$, while each vertical segment of $P$ is completely contained in a different subsegment from $S$. A grid path that connects $S$ has exactly $2|S|-2$ bends.

% We also say that a segment (or sub-segment) \emph{displays} the corresponding vertex or edge if it consists only of grid-edges witnessing that the vertex or edge is displayed.

%%%%%%%%%%%%%%%%%%%%%%%%%%%%%%%%%%%%%%%%%%%%%%%%%%%%%%%%%%%%
%%%%%%%%%%%%%%%%%%%%%%%%%%%%%%%%%%%%%%%%%%%%%%%%%%%%%%%%%%%%
%%%%%%%%%%%%%%%%%%%%%%%%%%%%%%%%%%%%%%%%%%%%%%%%%%%%%%%%%%%%

\section{Edge Clique Covers}\label{sec:cli}
In this section we present a general method to represent any graph with a number of bends depending on an edge clique cover. More precisely, let $\mathcal{C}$ be the graph class of cliques and their disjoint unions. A collection $C_1,\ldots,C_{\ell}$ of members of $\mathcal{C}$ is an \emph{edge clique cover} of $G$ if each $C_i$ can be mapped into $G$ by an injective homomorphism, so that each edge of $G$ is in the image of at least one $C_i$. Denote by $cl_g(G)$ the \emph{global clique covering number}, i.e., the minimum number of elements of $\mathcal{C}$ needed for an edge clique  cover of $G$. Moreover denote by $cl_{\ell}(G)$ the \emph{local clique covering number}, i.e., the minimum over all edge clique covers of $G$ of the maximum number of members containing the same vertex of $G$ in their image. One easily sees $cl_{\ell}(G)\leq cl_g(G)$. %These graph parameters embed into a more general context, see~\cite{Kna-12}.

\begin{thm}\label{thm:cli}
We have $b(G)\leq cl_g(G)-1$.
\end{thm}
\begin{proof}
Let $C_i$ for $1\leq i\leq k=cl_g(G)$ denote the members of the edge clique cover, extended by 1-cliques, such that each $C_i$ covers all vertices. We use staircases with $k-1$ bends to represent the vertices of $G$. Every $C_i$ consists of vertex disjoint cliques $S_i(j)$ covering all vertices for $1\leq j\leq \ell_{i}$. We associate with every vertex $v$ the vector $x_v$ with $x_v(i)=j$ iff $v\in S_i(j)$. We set $\ell_{-1}=0$ and $\ell_{0}=1$ and $x_v(0)=0$. The path $P(v)$ is defined by the coordinates of its start point, bends and end point $(p_1,p_2,\ldots,p_{k},p_{k+1})$. Define $p_i:=(x_v(i)+i-1+\sum_j^{(i-1)/2}\ell_{2j-1},x_v(i-1))$ if $i$ is odd and $(x_v(i-1),x_v(i)+i-1+\sum_j^{(i-2)/2}\ell_{2j})$ if $i$ is even.

Along the $i$-th segment, two paths run through the same grid-line if and only if the corresponding vertices are in the same clique in $C_i$. See Figure~\ref{fig:chromaticclique} for an illustration.

\begin{figure}[htb]
 \centering
 \psfrag{1}[cc][cc]{$1$}
 \psfrag{2}[cc][cc]{$2$}
 \psfrag{3}[cc][cc]{$3$}
 \psfrag{4}[cc][cc]{$4$}
 \psfrag{5}[cc][cc]{$5$}
 \psfrag{u}[cc][cc]{$P(u)$}
 \psfrag{v}[cc][cc]{$P(v)$}
 \includegraphics{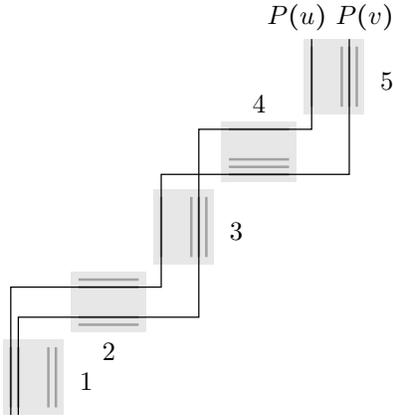}
 \caption{A $(k-1)$-bend representation based on an edge clique cover with $k$ unions of cliques: The grey blocks correspond to the members of the clique cover. Every clique in $C_i$ is assigned a grid-line within the $i$-th block. Paths are inserted as demonstrated by $P(u)$ and $P(v)$ according to the cliques they are in.}
 \label{fig:chromaticclique}
\end{figure}
\end{proof}

\begin{exmpl}As illustrated in Figure~\ref{fig:trianglegrid}, every induced subgraph of the triangular plane grid has global clique covering number at most $3$. We conclude that
 the bend-number of the triangular plane grid and all of its induced subgraphs is at most 2.
\end{exmpl}

\begin{figure}[htb] 
 \centering
 \includegraphics{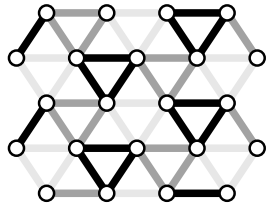}
 \caption{A subgraph $G$ of the triangular plane grid covered by $3$ unions of cliques. Hence $b(G) \leq 2$.}
 \label{fig:trianglegrid}
\end{figure}

Line graphs of bipartite graphs can be covered by two unions of cliques corresponding to the stars of the vertices of the first and the second bipartition class, respectively.

\begin{cor}
 Line graphs of bipartite graphs are single-bend graphs.
\end{cor}

%Note that the remark at the end of Section~\ref{sec:pla} implies that the triangular plane grid has no \emph{simple} 2-bend representation.

A proper edge coloring of $G$ is a special case of an edge clique cover of $G$ as in Theorem~\ref{thm:cli}. Hence the bend-number $b(G)$ is bounded by the edge chromatic number $\chi'(G)$. This implies a strengthening of a result of~\cite{Bie-10} concerning the maximal degree $\Delta$ of $G$:

\begin{cor}\label{cor:chi'}
If $\chi'(G)$ denotes the edge chromatic number of $G$, then $b(G)\leq\chi'(G)-1$. In particular Vizing's Theorem~\cite{Viz-64} yields $b(G)\leq\Delta$ for $G$ with maximum degree $\Delta$.
\end{cor}

\begin{quest}\label{quest:Delta}
 What is the maximum bend-number of a graph with maximum degree $\Delta$? By Theorem~\ref{thm:kmm} and Corollary~\ref{cor:chi'} it is between $\lceil \frac{\Delta}{2} \rceil$ and $\Delta$.
\end{quest}

We now provide a bound of $b(G)$ in terms of the local clique covering number.
\begin{thm}\label{thm:cli_max_deg}
We have $b(G) \leq 2cl_{\ell}(G)-2$.
\end{thm}
\begin{proof}
Consider an edge clique cover $C$ of $G$ such that each vertex is contained in at most $k$ members of $C$. Reserve parallel segments of vertical grid-lines, one for each clique in $C$, such that all these segments see each other. We represent every vertex $v$ by a path $P(v)$ connecting the vertical segments corresponding to the $k$ cliques containing $v$. The horizontal segments of $P(v)$ are such that no horizontal segments of paths in the representation intersect. See Figure~\ref{fig:degreeclique} for an illustration.
\begin{figure}[htb] 
 \psfrag{u}[cc][cc]{$P(v)$}
 \psfrag{v}[cc][cc]{$P(u)$}
 \centering
 \includegraphics{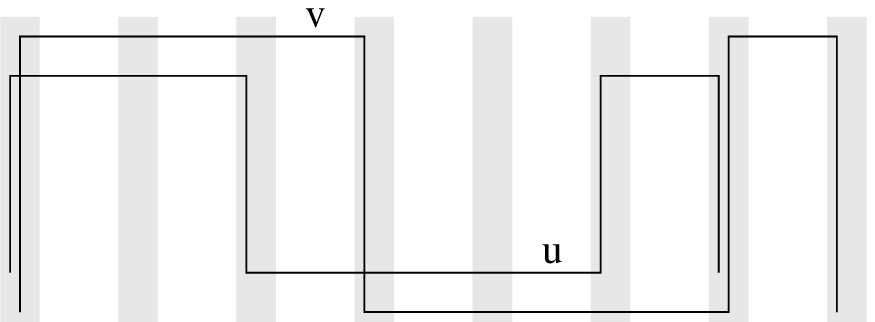}
 \caption{A $(2cl_{\ell}(G)-2)$-bend representation: Each vertical grey part is reserved for a member of $C$. Paths are inserted as demonstrated by $P(u)$ and $P(v)$ according to the cliques they are in.}
 \label{fig:degreeclique}
\end{figure}
\end{proof}

Line graphs have local clique covering number at most $2$, while the backwards direction does not hold, e.g., $K_5-e$. Hence Theorem~\ref{thm:cli_max_deg} generalizes a result of \cite{Bie-10}, stating that line graphs have bend-number at most 2. %\emph{Triangle graphs} have local clique covering number at most 3, see~\cite{Bag-04} for some classes of triangle graphs. We conclude in this case from Theorem~\ref{thm:cli_max_deg} that $b(G) \leq 4$. 
% For line graphs it is easy to see that $2$ bends are sometimes necessary.
% 
% \begin{conj}
%  The maximal bend-number attained on the class of graph with $cl_{\ell}\leq k$ is $2k-2$.
% \end{conj}

Line graphs are claw-free and in~\cite{Bie-10} it was asked whether also claw-free graphs have bend-number at most $2$. We feel that this is not the case.

\begin{conj}\label{conj:claw}
 There are claw-free graphs with arbitrary large bend-number.
\end{conj}

%%%%%%%%%%%%%%%%%%%%%%%%%%%%%%%%%%%%%%%%%%%%%%%%%%%%%%%%%%%%
%%%%%%%%%%%%%%%%%%%%%%%%%%%%%%%%%%%%%%%%%%%%%%%%%%%%%%%%%%%%
%%%%%%%%%%%%%%%%%%%%%%%%%%%%%%%%%%%%%%%%%%%%%%%%%%%%%%%%%%%%
\section{Complete bipartite graphs}\label{sec:bip}

In this section we consider the bend-number of complete bipartite graphs. The interval-number of $K_{m,n}$ has been determined independently by Harary and Trotter~\cite{Har-79} and Griggs and West~\cite{Gri-80}, its track-number was determined by Gy{\'a}rf{\'a}s and West~\cite{Gya-95}. Both values are given by simple closed formulas, namely $i(K_{m,n}) = \lceil \frac{mn+1}{m+n} \rceil$ and $t(K_{m,n}) = \lceil \frac{mn}{m+n-1} \rceil$. In contrast, a closed formula for the bend-number of complete bipartite graphs seems hard to obtain, as the results in this section illustrate.

It was shown by Biedl and Stern~\cite{Bie-10} that $b(K_{m,n}) \leq 2m-2$ for all $n$ and that this bound is attained when $n > 4m^4 - 8m^3 + 2m^2 + 2m$. Other than that the only bound on $b(K_{m,n})$ known before is $b(K_{m,n}) \leq \lceil\frac{\operatorname{max}\{m,n\}}{2}\rceil$, due to Asinowski and Suk~\cite{Asi-09}.

We present here two non-trivial lower bounds on $b(K_{m,n})$: The Lower-Bound-Lemma~I (Lemma~\ref{lem:low}) gives, depending on $n$, lower bounds on $b(K_{m,n})$ ranging from $\frac{m}{2}$ to $m-1$. It in particular implies that the construction of Asinowski and Suk~\cite{Asi-09} is tight in the case $m=n$, i.e., $b(K_{m,m}) = \lceil \frac{m}{2} \rceil$. This especially confirms a conjecture~\cite{Gol-09,Bie-10} stating that for every $k$ there is a graph $G$ with $b(G) = k$. Moreover, we present in Theorem~\ref{thm:kmm3} a new construction asserting that the Lower-Bound-Lemma~I is also tight at the other end of its range, i.e., $b(K_{m,n}) = m-1$ for all $n$ roughly between $m^2$ and $m^3/4$. Interestingly, we obtain that the maximal $n$ for which $b(K_{m,n}) \leq k$ jumps from $n \in \theta(m^2)$ to $n \in \theta(m^3)$ when going from $k=m-2$ to $k=m-1$.

Secondly, the Lower-Bound-Lemma~II (Lemma~\ref{lem:lower-bound-2}) gives, depending on $n$, lower bounds on $b(K_{m,n})$ ranging from $m-1$ to $2m-2$. It improves the bound on the minimal $n$ for which $b(K_{m,n}) = 2m-2$ to $m^4-2m^3+5m^2-4m+1$ (c.f. Theorem~\ref{thm:kmm4}). With another construction in Theorem~\ref{thm:m4} we show that our new bound is nearly tight, i.e., we show that $b(K_{m,n}) \leq 2m-3$ for $n \leq m^4 -2m^3 + \frac{5}{2}m^2 -2m -4$, which disproves a conjecture of Biedl and Stern~\cite{Bie-10}, who suspected that this can be achieved only for $n \in \mathcal{O}(m^2)$. The Lower-Bound-Lemma~II is almost tight at the lower, as well as at the upper end of its range.

To summarize, we determine in this section $b(K_{m,n})$ for several pairs of $m,n$ and present two lower bounds for the general case. Our results are illustrated in Figure~\ref{fig:general-Kmn}, where we sketch a region that contains the graph of $b(K_{m,n})$ seen as a function of $n$ for fixed $m$.

\begin{figure}[htb]
 \centering
 \psfrag{x1}[rc][rc]{$\frac{1}{2}m$}
 \psfrag{x2}[rc][rc]{$m-1$}
 \psfrag{x3}[rc][rc]{$\frac{3}{2}m$}
 \psfrag{x4}[rc][rc]{$2m-3$}
 \psfrag{x5}[rc][rc]{$2m-2$}
 \psfrag{y1}[rc][lc][1][25]{$m$}
 \psfrag{y2}[rc][lc][1][25]{$(m-1)^2$}
 \psfrag{y3}[rc][lc][1][25]{$\frac{1}{4}m^3-\frac{1}{2}m^2-m+4$}
 \psfrag{y4}[rc][lc][1][25]{$\frac{1}{4}m^3+\frac{1}{2}m^2+3m$}
 \psfrag{y5}[rc][lc][1][25]{$m^4-2m^3+\frac{5}{2}m^2-2m-4$}
 \psfrag{y6}[rc][lc][1][25]{$m^4-2m^3+5m^2-4m+1$}
 \psfrag{U1}[rc][rc]{Thm.~\ref{thm:kmm3}}
 \psfrag{U2}[rc][rc]{Thm.~\ref{thm:m4}}
 \psfrag{L1}[lc][lc]{Thm.~\ref{thm:kmm}}
 \psfrag{L2}[lc][lc]{Lower-Bound-Lemma I}
 \psfrag{L3}[lc][lc]{Lower-Bound-Lemma II}
 \psfrag{L4}[lc][lc]{Thm.~\ref{thm:kmm4}}
 \includegraphics{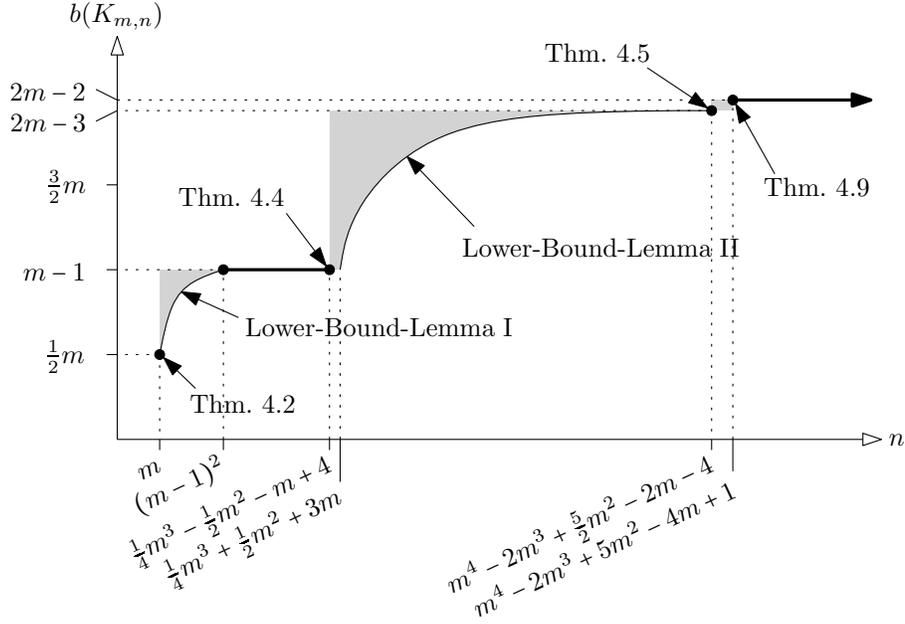}
 \caption{Illustration of the lower and upper bounds on $b(K_{m,n})$ presented in this section. The grey region contains the graph of the function $n \mapsto b(K_{m,n})$ for fixed $m$. Filled circles and solid lines correspond to exact values.}
 \label{fig:general-Kmn}
\end{figure}

The bend-number of $K_{2,n}$ has been determined for all $n$ in~\cite{Asi-09}: $b(K_{2,n})=2$ iff $n\geq 5$, $b(K_{2,n})=1$ iff $2\leq n\leq 4$ and $b(K_{2,n})=0$ iff $n\leq 1$. Throughout this paper we denote the bipartition classes of $K_{m,n}$ by $A=\{a_1,\ldots,a_m\}$ and $B=\{b_1,\ldots,b_n\}$ and will always assume $3\leq m\leq n$.

We start with the first lower bound on the bend-number of $K_{m,n}$. In the lemma below, as well as in the second lower bound in Lemma~\ref{lem:lower-bound-2}, we present an inequality with variables $m$, $n$ and $k$ that is valid for every $k$-bend representation of $K_{m,n}$. Fixing $m$ and $n$, we obtain a lower bound on $k = b(K_{m,n})$; Fixing $m$ and $k$ we obtain an upper bound on the maximum $n$ such that $b(K_{m,n}) \leq k$.

\begin{lem}[\textbf{Lower-Bound-Lemma~I}]\label{lem:low}
For every $k$-bend representation of $K_{m,n}$ we have $$(k+1)(m+n) \geq mn + \sqrt{2k(m+n)}.$$
\end{lem}

\begin{proof}
 Fix a $k$-bend representation of $K_{m,n}$ and assume w.lo.g. that every path has exactly $k$ bends. Let $\mathcal{L}$ be the set of all grid-lines that support any segment of any path. Further, let $\mathcal{G}$ be the set of grid-points that support a bend of any path in the representation, and $\mathcal{G}_3 \subseteq \mathcal{G}$ be those grid-points that support at least $3$ bends.

 Now, look at the rightmost or topmost grid-edge of each of the $k+1$ segments of every path $P(v)$. If this grid-edge is shared by another path $P(w)$ (there can be only one, because the graph is bipartite), we assign $v$ to the edge $\{v,w\}$ in the graph. This way
 \begin{itemize}
  \item every vertex is assigned to at most $k+1$ edges,
  \item every edge is assigned at least once, 
  \item at the rightmost (topmost) grid-edge of every line in $\mathcal{L}$ either no assignment is done or an edge is assigned twice, and
  \item at the left and bottom grid-edge of every point in $\mathcal{G}_3$ either no assignment is done or an edge is assigned twice.
 \end{itemize}
 Hence we have $(k+1)|V(K_{m,n})| \geq |E(K_{m,n})| + |\mathcal{L}| + 2|\mathcal{G}_3|$, i.e.,
 \begin{equation}
  (k+1)(m+n) \geq mn + |\mathcal{L}| + 2|\mathcal{G}_3|.\label{eqn:L-up}
 \end{equation}
 Since every path has exactly $k$ bends and every grid-point supports at most $4$ of them, we get 
 $|\mathcal{G}| \geq \frac{k(m+n) - 4|\mathcal{G}_3|}{2} + |\mathcal{G}_3| = \frac{k(m+n)}{2}-|\mathcal{G}_3|$. Now if we have $\ell_h$ horizontal and $\ell_v$ vertical lines in $\mathcal{L}$, and these lines cross in at least $|\mathcal{G}|$ grid-points, then $|\mathcal{L}|=\ell_h+\ell_v$ with $\ell_h\cdot\ell_v\geq |\mathcal{G}|$. The sum is minimized by $\ell_h=\ell_v=\sqrt{|\mathcal{G}|}$. Hence, 
 \begin{equation}
  |\mathcal{L}|\geq 2\sqrt{|\mathcal{G}|} \geq \sqrt{2k(m+n) - 4|\mathcal{G}_3|} \geq \sqrt{2k(m+n)} - 2\sqrt{|\mathcal{G}_3|}.\label{eqn:L-low}
 \end{equation}
 Putting~\eqref{eqn:L-up} and~\eqref{eqn:L-low} together we obtain $$(k+1)(m+n) \geq mn + \sqrt{2k(m+n)} - 2\sqrt{|\mathcal{G}_3|} + 2|\mathcal{G}_3| \geq mn + \sqrt{2k(m+n)},$$
 which completes the proof.
\end{proof}

The Lower-Bound-Lemma~I is meaningful only for $k \leq m-2$. Indeed if $k \geq m-1$, then the inequality in the Lower-Bound-Lemma~I is valid for all $n$. The reason is that for $k \geq m-1$ the paths for $B$ could have each edge-intersection on a different segment. However, the Lower-Bound-Lemma~I gives non-trivial lower bounds on $b(K_{m,n})$ ranging from $\frac{m}{2}$ to $m-1$ depending on $n$. These bounds turn out to be best-possible at both ends of this range, i.e., we can determine $b(K_{m,n})$ exactly for some particular pairs of $m$ and $n$ (c.f. Theorem~\ref{thm:kmm} and Theorem~\ref{thm:kmm3}).

\begin{thm}\label{thm:kmm}
For all $m\geq 3$ we have $b(K_{m,m})=\lceil\frac{m}{2}\rceil$.
\end{thm}
\begin{proof}
As mentioned above in~\cite{Asi-09} it is shown that $b(K_{m,n}) \leq \lceil \max(m,n)/2 \rceil$. For equality we will prove that $K_{m,m}$ cannot be represented with less than $\lceil\frac{m}{2}\rceil$ bends. We use the Lower-Bound-Lemma~I with $m=n$ and $k=\frac{m-1}{2}\geq\lceil\frac{m}{2}\rceil-1$. Bringing everything on the left-hand-side we obtain $$(m-\sqrt{2m(m-1)})/2m\geq 0,$$ which is a contradiction for $m\geq 3$.
\end{proof}

With Theorem~\ref{thm:kmm} we can confirm a conjecture of~\cite{Gol-09}. 
\begin{cor}
 For every $k \geq 0$ there is a graph $G$ with $b(G) = k$.
\end{cor}

%%%%%%%%%%%%%%%%%%%%%%%%%%%%%%%%%%%%%%%%%
%% N E W   V E R S I O N   S T A R T S %%
%%%%%%%%%%%%%%%%%%%%%%%%%%%%%%%%%%%%%%%%%

From the Lower-Bound-Lemma~I we get that the bend-number of $K_{m,n}$ is at least $m-1$ if $n \geq (m-1)^2$. Next we show that this bound is tight, that is, we can find a $(m-1)$-bend representation of $K_{m,n}$ even for $n \approx \frac{m^3}{4}$ (c.f. Theorem~\ref{thm:kmm3}).

Our construction (as well as the one in Theorem~\ref{thm:m4}) is based on the fact that two $(2j-1)$-bend paths can cross in $\ell(\ell+1)$ points as indicated in Figure~\ref{fig:cross}. In Lemma~\ref{lem:pnt} we will show that this is best-possible. Formally, $P_1$ may be described as starting at $(0,0)$ with a horizontal segment to the right, followed by a downwards segment, and then alternate left and right and up and down. The length of the $i$-th horizontal segment is $2j+3-2i$ and of the $i$-th vertical segment is $2i$. The path $P_2$ is obtained by rotating $P_1$ by $180$ degrees and translating the starting point to $(2j+2,-1)$. We call the pair $(P_1,P_2)$ a \emph{pretzel}.

\begin{figure}[htb]
 \centering
 \psfrag{5}[bc][bc]{$j=5$}
 \psfrag{6}[bc][bc]{$j=6$}
 \psfrag{7}[bc][bc]{$j=7$}
 \includegraphics{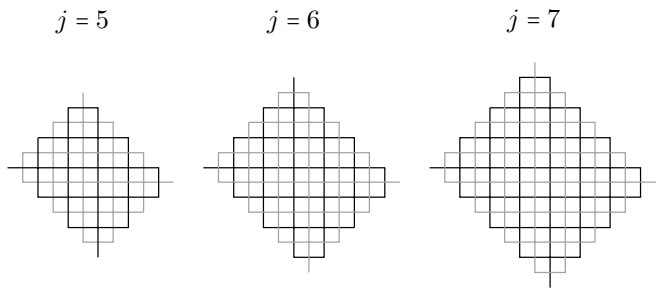}
 \caption{Pretzels: Two $(2j-1)$-bend paths can cross in $j(j+1)$ points.}
 \label{fig:cross}
\end{figure}

For the next two constructions we ``blow up'' the pretzels in Figure~\ref{fig:cross}: We split the vertices in $A$ evenly into $A_1=(a_1^1, \ldots, a^1_{\lfloor\frac{m}{2}\rfloor})$ and $A_2=(a_1^2, \ldots, a^2_{\lceil\frac{m}{2}\rceil})$. Then, for fixed $j$ we construct a set of $m$ $(2j-1)$-bend paths such that every path representing a vertex from $A_i$ looks like $P_i$ just with changed segment-lengths ($i=1,2$). More precisely, $P(a_1^1)$ has starting point $(1,1)$, its $i$-th horizontal segment is of length $\lfloor\frac{m}{2}\rfloor(m-i)+\lceil\frac{m}{2}\rceil(m-i+1)$ and its $i$-th vertical segment is of length $\lfloor\frac{m}{2}\rfloor(i-1)+\lceil\frac{m}{2}\rceil(i+1)$. 
For $h = 2,\ldots,\lfloor\frac{m}{2}\rfloor$, $P(a_h^1)$ has starting point $(h,h)$, its $i$-th horizontal segment is $2$ units shorter than the $i$-th horizontal segment of $P(a_{h-1}^1)$ and its $i$-th vertical segment is $2$ units longer than the $i$-th vertical segment of $P(a_{h-1}^1)$. In the end, we enlarge all starting segments such that $P(a_h^1)$ has starting point $(1,h)$ ($h = 1,\ldots,\lfloor\frac{m}{2}\rfloor$) and all ending segments such that all endpoints lie on one line.

The paths representing $A_2$ arise from $P_2$ in an analogous way by interchanging floor and ceiling functions in the length formulas. Their starting points are at $(\lfloor\frac{m}{2}\rfloor(m-1)+\lceil\frac{m}{2}\rceil m+1,1-h)$ ($h = 1,\ldots,\lceil\frac{m}{2}\rceil$). The blown-up pretzel for $m=6$ and $j=3$ is depicted in Figure~\ref{fig:multicross}.

Crossings in a blown-up pretzel are either formed by paths from the same $A_i$ (type~$1_i$) or from $A_1$ and $A_2$ (type~$2$). Type~$1_i$ crossings form \emph{blocks} of size $|A_i|\times |A_i|$ and the type~$2$ crossings come in \emph{blocks} of size $|A_1|\times |A_2|$. Some blocks of type~$1_i$ are located around one bend of the corresponding paths. These blocks have triangular instead of quadrangular shape. All blocks form a checkerboard pattern. In Figure~\ref{fig:multicross} the type~$2$ blocks are highlighted by grey boxes.

We group the blocks in a blown-up pretzel into four quadrants as illustrated in Figure~\ref{fig:multicross}. I.e., we choose a horizontal and a vertical line that separate the horizontal and vertical ends of paths for $A_1$ from those for $A_2$, respectively.

\begin{figure}[htb]
 \centering
 \psfrag{2}[cc][cc]{$\lceil\frac{m}{2}\rceil$}
 \psfrag{0}[cc][cc]{$1,2$}
\psfrag{n}[cc][cc]{$m$}
 \includegraphics{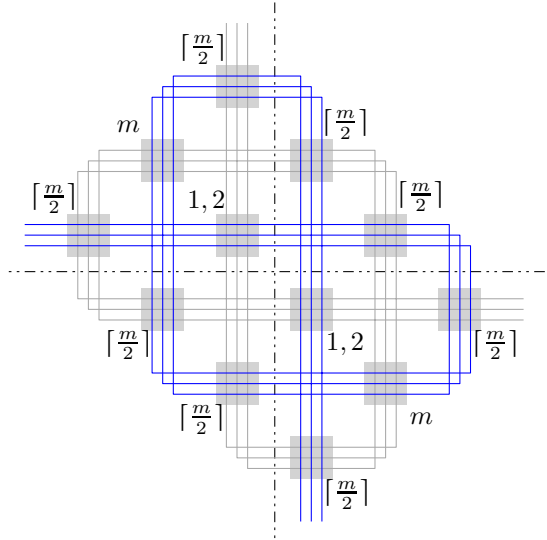}
 \caption{A blown up pretzel for $m=6$ and $j=3$. The dashed lines indicate the four quadrants used in the proof of Theorem~\ref{thm:kmm3}. Type~$2$ blocks are highlighted by grey boxes. Next to the boxes the numbers of staircases that can be introduced there are noted. For better readability grid-edges between distinct blocks are drawn longer.}
 \label{fig:multicross}
\end{figure}

\begin{thm}\label{thm:kmm3}
If $m$ is even, then $b(K_{m,{1}/{4}m^3-{1}/{2}m^2-m+4}) = m-1$. If $m$ is odd, then $b(K_{m,{1}/{4}m^3-m^2+{3}/{4}m}) = m-1$.
\end{thm}
\begin{proof}
 We represent the bipartition class $A$ by a blown up pretzel, where $j:=\lfloor\frac{m}{2}\rfloor$. That is, if $m$ is odd we use one bend less than allowed. We analyze the blocks of the blown-up pretzel. Inside a quadrant, the order in which the paths $P(a^i_h)$ ($h=1,\ldots,\lceil \frac{m}{2} \rceil$, $i=1,2$) appear horizontally and vertically in the blocks is always the same. See Figure~\ref{fig:kmm3} for an example.

 Next we explain how to represent a subset of vertices of $B$ as small staircases with $m-1$ bends inside the first (that is the top-right) quadrant. Pick a block of type~$2$ inside the quadrant. We use staircases that go from lower left to upper right. Half of the paths start with a horizontal segment and the other half with a vertical segment. All segments have length~$1$. As illustrated in Figure~\ref{fig:kmm3}, it is possible to put $m$ paths into one block of type~$2$. If $m$ is even, two paths lie completely within the blocks, $\frac{m-2}{2}$ paths have segments lying in a triangular sector in the type~$1_i$ blocks above and to the left, and $\frac{m-2}{2}$ paths have segments in a triangular sector in the type~$1_i$ blocks below and to the right. If $m$ is odd, it is one path that lies completely within the blocks, $\frac{m-1}{2}$ paths have segments in a triangular sector in the type~$1_i$ blocks above and to the left, and $\frac{m-1}{2}$ paths have segments in a triangular sector in the type~$1_i$ blocks below and to the right. Note that the type~$1_i$ blocks above and to the right may be triangular, i.e., correspond to a set of bends and that in this case still every segment establishes an edge-intersection. This case is depicted in Figure~\ref{fig:kmm3}.

 The staircases for each type~$2$ block are completely contained in a square area as illustrated in Figure~\ref{fig:kmm3}. Since squares for distinct blocks are disjoint, no two staircases have an edge-intersection. Finally, we remove all staircases that are crossing the coordinate axes, i.e., keep only those completely contained in the first quadrant.

 \begin{figure}[htb]
  \centering
  \psfrag{1}[cc][cc]{$P(a^1_1)$}
  \psfrag{2}[cc][cc]{$P(a^1_2)$}
  \psfrag{3}[cc][cc]{$P(a^1_3)$}
  \psfrag{4}[cc][cc]{$P(a^1_4)$}
  \psfrag{a}[cc][cc]{$P(a^2_1)$}
  \psfrag{b}[cc][cc]{$P(a^2_2)$}
  \psfrag{c}[cc][cc]{$P(a^2_3)$}
  \psfrag{d}[cc][cc]{$P(a^2_4)$}
  \psfrag{...}[cc][cc]{$\cdots$}
  \includegraphics{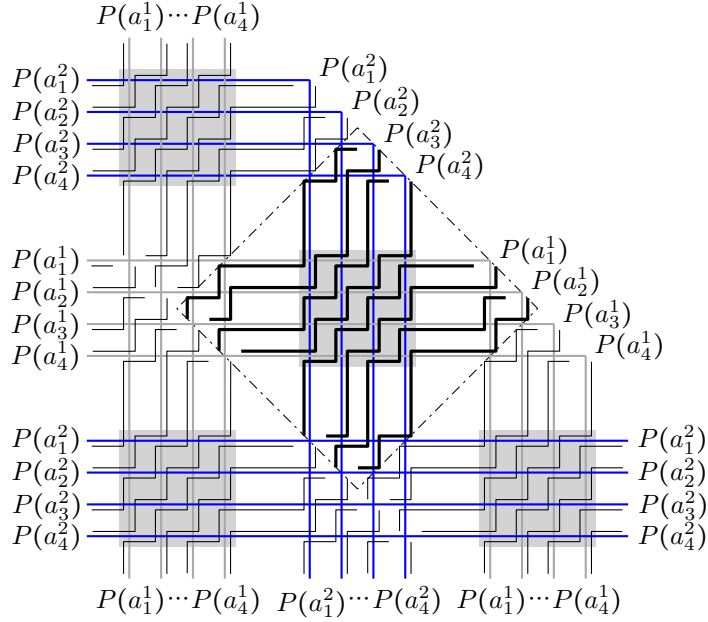}
  \caption{Inserting $m$ staircases for each type~$2$ block in the first quadrant. For better readability grid-edges between distinct blocks are drawn longer.}
  \label{fig:kmm3}
 \end{figure}

 Next we extend the above construction to all four quadrants by mirroring the staircase paths along the coordinate axes. Since type~$2$ blocks correspond to crossings between $P_1$ and $P_2$ in the underlying pretzel, there are $\lfloor\frac{m}{2}\rfloor(\lfloor\frac{m}{2}\rfloor+1)$ of them. Note further that at type~$2$ blocks close to the coordinate axes we have removed some staircases that crossed the coordinate axes. In particular we have introduced only $\lceil\frac{m}{2}\rceil$ staircases at a type~$2$ block close to one coordinate axis. There are $2(2j-1)=4\lfloor\frac{m}{2}\rfloor-2$ of these blocks. At the two type~$2$ blocks next to two coordinate axes all but $1$, respectively $2$, paths have been removed for $m$ odd, respectively $m$ even. In Figure~\ref{fig:multicross} we summarize the number of staircases introduced at each type~$2$ block for $m=6$ and $j = \lfloor\frac{m}{2}\rfloor = 3$.

 In total for odd $m$ we obtain $(\lfloor\frac{m}{2}\rfloor(\lfloor\frac{m}{2}\rfloor+1)-(4\lfloor\frac{m}{2}\rfloor))m+(4\lfloor\frac{m}{2}\rfloor-2)\lceil\frac{m}{2}\rceil+2$ staircases, which simplifies to $\frac{1}{4}m^3-m^2+\frac{3}{4}m$. For even $m$ we get the same formula plus $2$, which simplifies to $\frac{1}{4}m^3-\frac{1}{2}m^2-m+4$. Hence we have constructed a $(m-1)$-bend representation of $K_{m,{1}/{4}m^3-{1}/{2}m^2-m+4}$ for $m$ even and $K_{m,{1}/{4}m^3-m^2+{3}/{4}m}$ for $m$ odd, which completes the proof.
\end{proof}

\begin{rem}
 By a straight-forward application of the Lower-Bound-Lemma I one sees that already for much smaller $n$ than in the above theorem $m-1$ bends are necessary, i.e., $b(K_{m,(m-1)^2})\geq m-1$ for all $m\geq 3$.
\end{rem}

It is known that the bend-number of $K_{m,n}$ can be as large as $2m-2$~\cite{Asi-09}, but not larger~\cite{Gol-09}. Let $n^*$ denote the maximal $n$ for which $b(K_{m,n}) < 2m-2$. Asinowski and Suk~\cite{Asi-09} proved that $n^* \in \mathcal{O}(m^m)$. This was later improved by Biedl and Stern~\cite{Bie-10} to $n^* \leq 4m^2-8m^3+2m^2+2m$, who also conjectured that $n^* \in \mathcal{O}(m^2)$. Note that Theorem~\ref{thm:kmm3} above implies that $n^* \geq \frac{1}{4}m^3-m^2$, and hence disproves this conjecture.

Next we show that $n^* \geq m^4 - 2m^3 + \frac{5}{2}m^2-2m-4$, i.e., we find a $(2m-3)$-bend representation of $b(K_{m,n})$ with $n = m^4 - 2m^3 + \frac{5}{2}m^2-2m-4$, see Theorem~\ref{thm:m4}. Afterwards, we present the Lower-Bound-Lemma~II, a special case of which improves the upper bound on $n^*$ to $m^4 - 2m^3 + 5m^2 -4m$, see Theorem~\ref{thm:kmm4}. Hence we have narrowed $n^*$ down to an interval of length less than $2.5m^2$.

\begin{thm}\label{thm:m4}
 If $n\leq m^4-2m^3+\frac{5}{2}m^2-2m-4$ then $b(K_{m,n})\leq 2m-3$.
\end{thm}
\begin{proof}
 Again, we use blown-up pretzels. This time, we set $j:=m-1$, i.e., we use $(2m-3)$-bend paths. We scale the blown-up pretzel by a factor of $2m^2+3$, such that between any two horizontal (vertical) segments we have $2m^2+2$ horizontal (vertical) grid-lines.

 We start by defining a single-bend path, that is a $1$-bend path, $P(b)$ for every vertex in $b\in B$. Every $P(b)$ has an edge-intersection with two paths for $A$. In a second step every $P(b)$ will be extended to a $(2m-3)$-bend path that has an edge-intersection with $P(a)$ for \emph{every} $a \in A$. Let us consider the first, that is the top-right, quadrant only. The paths for the $q$-th quadrant are defined analogously by rotating the entire representation by $(q-1)\frac{\pi}{2}$ in counterclockwise direction, $q=1,2,3,4$.

 We number the vertical segments of paths for $A$ that intersect the first quadrant by $v_1,v_2,v_3,\ldots$ according to their left-to-right order, i.e., by increasing $x$-coordinate. Similarly, we number the horizontal segments that intersect the first quadrant by $h_1,h_2,h_3,\ldots$ according to their top-down order, that is by decreasing $y$-coordinate. Consider a crossing $v_i \cap h_k$ between two segments corresponding to two distinct paths. We define two single-bend paths that both have a bend at the crossing point $p$. The corresponding vertical segments have length~$2i$ and lie on different sides of $p$. The two corresponding horizontal segments have length~$\max(1,2(\lceil\frac{m}{2}\rceil-k))$ and lie on different sides of $p$, too. In particular, the horizontal segments have length more than $1$ only if $k < \lceil\frac{m}{2}\rceil$. See Figure~\ref{fig:kmm4} for an example.

 \begin{figure}[htb]
  \centering
  \psfrag{1}[cc][cc]{$P(a^1_1)$}
  \psfrag{2}[cc][cc]{$P(a^1_2)$}
  \psfrag{a}[cc][cc]{$P(a^2_1)$}
  \psfrag{b}[cc][cc]{$P(a^2_2)$}
  \includegraphics{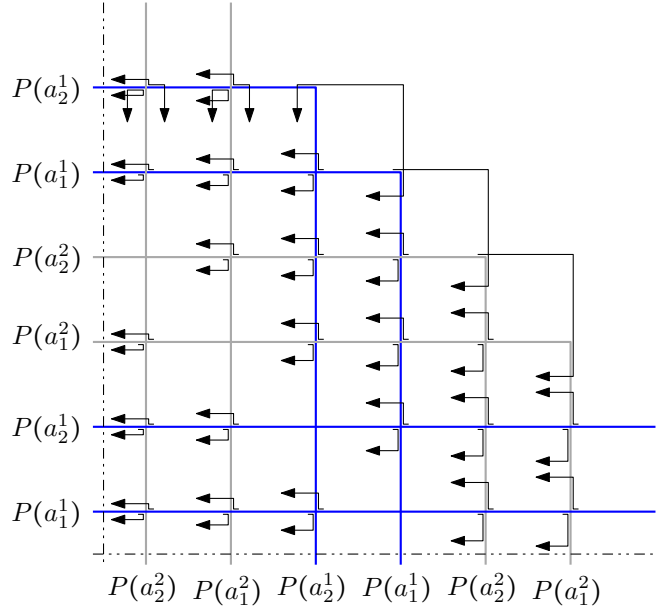}
  \caption{Interlacing two single-bend paths for every crossing and one single-bend path at each but the left-most bend. In a second step these paths are extended by snake paths in the direction indicated by the arrows.}
  \label{fig:kmm4}
 \end{figure}

 We introduce another set of single-bend paths for each but the left-most bend in the first quadrant. Let the bend be the top-end of the vertical segment $v_i$. We define a single-bend path edge-intersecting $v_i$ and the horizontal segments $h_k$ that has a bend with $v_{i-1}$. The horizontal and vertical segment of the single-bend path has length $2m^2+3+\max(1,\lceil\frac{m}{2}\rceil-k)$ and $2m^2 + 3 + 2i$, respectively. In particular, each vertical edge-intersection has length $2i$ and a horizontal edge-intersection has length more than $1$ only if $k < \lceil\frac{m}{2}\rceil$. See Figure~\ref{fig:kmm4} for an example.
 
 Next we extend every single-bend path $P$ at its vertical end. Let $\ell(P)$ be the horizontal grid-line through the vertical end of $P$. From the initial scaling follows that there is at least one further horizontal grid-line between every two $\ell(P)$, as well as between every $\ell(P)$ and every horizontal segment of the blown-up pretzel. Now, for every vertex $a\in A$, except for the two whose paths are already edge-intersecting $P$, consider the \emph{rightmost} vertical segment of $P(a)$ in the \emph{fourth}, that is the top-left, quadrant that crosses $\ell(P)$. We extend $P$ by a snake path connecting all these segments, i.e, every new vertical segment of $P$ is contained in a segment of some $P(a)$, has length~$1$ and its lower end lies on $\ell(P)$. In Figure~\ref{fig:kmm4} the first segment of each snake is indicated by an arrow pointing to the left.

 For some $1$-bend paths $P$ not all paths for $A$ cross the horizontal line $\ell(P)$. However, this is the case only if $P$ has an edge-intersection with $h_k$ for $k < \lceil\frac{m}{2}\rceil$. We let the snake extension for these paths edge-intersect only those paths $P(a)$ that indeed cross $\ell(P)$. In order to establish the remaining edge-intersections we extend the $1$-bend path $P$ at its left-end, too. More precisely, let $\ell'(P)$ be the vertical grid-line through the left-end of $P$. For each path $P(a)$ that has no edge-intersection with $P$ yet, consider the \emph{topmost} horizontal segment of $P(a)$ within the \emph{first} quadrant. We extend $P$ by a snake path of which every new horizontal segment is contained in such a topmost segment, has length~$1$ and whose left end lies on $\ell'(P)$.

 \medskip 

 We claim that we have defined a $(2m-3)$-bend representation of $K_{m,n}$ with $n = m^4-2m^3+\frac{5}{2}m^2-2m-4$. It is easy to check that every path $P(b)$ for a vertex $b$ in $B$ has $2m-3$ bends and edge-intersects all paths for $A$. Consider one quadrant $Q$, say the first. We consider all horizontal edge-intersections in $Q$: Most of the single-bend paths that have been introduced within $Q$ have a horizontal edge-intersection that has length~$1$ and lies next to a vertical segment of some path for $A$. There exist horizontal edge-intersections that are longer, but only within the top-most $\lceil\frac{m}{2}\rceil$ horizontal segments. Secondly, there are horizontal edge-intersections in $Q$ corresponding to snake extensions of single-bend paths in the second quadrant. Those edge-intersections consist of only one grid-edge that does \emph{not} lie next to a vertical segment of some path for $A$. Moreover, these edge-intersection are contained in the bottommost horizontal segments in $Q$ of each $P(a)$, i.e., each is contained in some $h_k$ for $k > \lceil\frac{m}{2}\rceil$. Finally, there are some edge-intersections corresponding to snake extensions of single-bend paths in $Q$. These edge-intersections consist of only one grid-edge that does \emph{not} lie next to a vertical segment of some path for $A$, and are chosen to be top-most, i.e., each is contained some $h_k$ for $k \leq \lceil\frac{m}{2}\rceil$.

 It follows that no two paths for $B$ have a horizontal edge-intersection in the first quadrant. An analogous reasoning holds for vertical edge-intersections, as well as for the other quadrants.

 It remains to determine how many paths for $B$ have been introduced. The paths $P_1,P_2$ in the underlying pretzel have $\binom{j}{2}$ self-crossings each, while there are $j(j+1)$ crossings between $P_1$ and $P_2$. Moreover, in the blown-up pretzel at every bend of one $P_i$ each pair of paths crosses ($i=1,2$). Thus, two paths for vertices in the same $A_i$ cross $2\binom{j}{2}+2j-1 = j(j+1)-1$ times, and two paths, one for a vertex in $A_1$ and the other for a vertex in $A_2$, cross $j(j+1)$ times. In terms of $m$, every pair of paths for the same $A_i$ crosses $m(m-1)-1$ times, and every pair of paths for distinct $A_i$ crosses $m(m-1)$ times. Hence the total number of crossings between paths representing vertices in $A$ is given by $\lfloor\frac{m}{2}\rfloor\lceil\frac{m}{2}\rceil m(m-1)+(\binom{\lceil m/2\rceil}{2}+\binom{\lfloor m/2\rfloor}{2})(m(m-1)-1)$.

 We have introduced two paths for every crossing in the blown-up pretzel and one path for all but one bend in each quadrant. Thus we have constructed a $(2m-3)$-bend representation of $K_{m,n}$ with $n = 2\lfloor\frac{m}{2}\rfloor\lceil\frac{m}{2}\rceil m(m-1)+ 2(\binom{\lceil m/2\rceil}{2}+\binom{\lfloor m/2\rfloor}{2})(m(m-1)-1) + m(2m-3) - 4 = \lfloor m^4-2m^3+\frac{5}{2}m^2-2m-4\rfloor$.
\end{proof}

Next we present a second lower bound on $b(K_{m,n})$ that is weaker than the Lower-Bound-Lemma~I for $n < m^2$, but stronger for $n > m^3$.

\begin{lem}\label{lem:c-inequality}
 In a $k$-bend representation of $K_{m,n}$ let $c$ denote the total number of crossings between the paths $P(a_1),\ldots,P(a_m)$. Then we have:
 \begin{displaymath}
  n(2m - k -2) \leq 2c + 2(k+1)m
 \end{displaymath}
\end{lem}
\begin{proof}
 Fix a $k$-bend representation of $K_{m,n}$. We associate every path $P(b)$, $b\in B$, with some crossings between paths for $A$ and some endpoints of segments of paths for $A$. Every crossing will be associated with at most two paths for $B$, and every endpoint of a segment with at most one path for $B$.

 Consider a vertex $b \in B$ and its path $P(b)$. Let $l$ be the number of segments of $P(b)$ on which \emph{no} edge-intersection occurs. Going along $P(b)$ we obtain a total order for the edge-intersections $P(b) \cap P(a)$ for all $a \in A$. For two edge-intersections $P(b)\cap P(a)$ and $P(b)\cap P(a')$ that appear \emph{consecutive} in this order let $s$ and $s'$ be the corresponding segments, respectively.

 If $s$ and $s'$ lie on the \emph{same} segment of $P(b)$, then we associate with $P(b)$ the two endpoints of $s$ and $s'$ with no edge-intersection between them. Since $b$ has degree $m$, $m$ edge-intersections occur on $k+1-l$ segments of $P(b)$. Thus we associate with $P(b)$ in this first step at least $2(m - (k+1-l)) = 2m - 2(k+1) + 2l$ endpoints of segments of paths for $A$.

 Now assume that $s$ and $s'$ lie on two \emph{consecutive} segments of $P(b)$. If $s$ and $s'$ cross at the corresponding bend of $P(b)$, we associate this crossing with $P(b)$. Otherwise at least one endpoint of $s$ or $s'$ lies on $P(b)$ and has no edge-intersection between itself and the bend of $P(b)$. We associate $P(b)$ with this endpoint. Since $P(b)$ has edge-intersections on $k+1-l$ of its $k+1$ segments we associate with $P(b)$ in this second step at least $k - 2l$ crossings between paths for $A$ or endpoints of segments of paths for $A$.

 Note that every endpoint of a segment of a path for $A$ is associated with at most one $P(b)$ and every  crossing between paths for $A$ is associated with at most two paths for $B$. Hence in total we have associated with all the paths for $B$ at least $n(2m - 2(k+1) + 2l + k - 2l) = n(2m - k -2)$ crossings and endpoints of paths for $A$. On the other hand every $P(a)$ with $a\in A$ has only $2(k+1)$ endpoints of segments and the total number of crossings between paths for $A$ is $c$. Thus $n(2m - k -2) \leq 2c + 2(k+1)m$, which completes the proof.
\end{proof}

We want to use the inequality in Lemma~\ref{lem:c-inequality} as a lower bound on $k$ for fixed $m$ and $n$. Therefore it remains to prove an upper bound on $c$, that is on the number of times $m$ $k$-bend paths can cross each other. We content ourselves with bounding the number of times two $k$-bend paths can cross in case $k$ is odd.

\begin{lem}\label{lem:pnt}
 Two $(2j-1)$-bend paths cross in at most $j(j + 1)$ points.
\end{lem}
\begin{proof}
 Consider two given $(2j-1)$-bend paths $P$ and $P'$. Both have exactly $j$ horizontal and $j$ vertical segments. We color the vertical segments of $P$ and the horizontal segments of $P'$ blue and the remaining segments red. Now every crossing is monochromatic. We partition the pairs of segments that have the same color but come from different paths into four sets. Set $\mathcal{B}$ contains all blue pairs that \emph{do} cross and $\overline{\mathcal{B}}$ all blue pairs that \emph{do not} cross. Similarly $\mathcal{R}$ and $\overline{\mathcal{R}}$ are defined for red segments. Along each path we index the segments, starting with its blue end, i.e., $s_1$ and $s'_1$ are blue and $P = (s_1,\ldots,s_{2j})$ and $P' = (s'_1,\ldots,s'_{2j})$. 

 Consider a blue crossing $\{s_i,s'_h\} \in \mathcal{B}$ and the grid-line $\ell$ containing $s_i$. Each of $s_{i-1},s_{i+1},s'_{h-1},s'_{h+1}$, if existent, is red and lies completely on one side of $\ell$. Moreover $s'_{h-1}$ and $s'_{h+1}$ cannot lie on the same side since $s'_h$ crosses $\ell$. Now consider $s_{i-1}$ (or $s_{i+1}$) and the segment of $P'$ on the other side of $\ell$. This pair evidently is in $\overline{\mathcal{R}}$. This way we associate up to two red non-crossings with every blue crossing, even if there are more (see Figure~\ref{fig:non-cross}).

 \begin{figure}[htb]
  \centering
  \psfrag{w}[cc][cc]{$s'_1$}
  \psfrag{w2}[cc][cc]{$s'_2$}
  \psfrag{w1}[cc][cc]{$s'_{h-1}$}
  \psfrag{w3}[cc][cc]{$s'_{h+1}$}
  \psfrag{v1}[cc][cc]{$s_{i-1}$}
  \psfrag{v3}[cc][cc]{$s_{i+1}$}
  \psfrag{l}[cc][cc]{$\ell$}
  \includegraphics{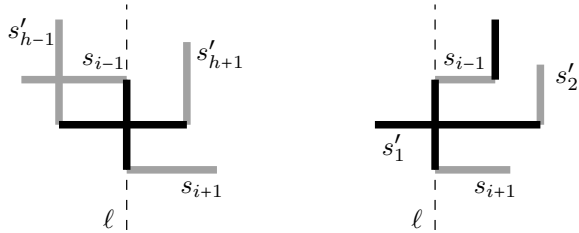}
  \caption{A blue crossing is associated with every pair of red segments from different sides of $\ell$: The blue crossing on the left is associated with $\{s_{i-1},s'_{h+1}\}$ and $\{s_{i+1},s'_{h-1}\}$. The blue crossing on the right is associated with no red non-crossing.}
  \label{fig:non-cross}
 \end{figure}

 Next we partition $\mathcal{B}$ in two ways. First, divide $\mathcal{B}$ into $\mathcal{B}_0, \mathcal{B}_1$, and $\mathcal{B}_2$ according to the number of red non-crossings, the blue crossings are associated with in the above way. Second, we write $\mathcal{B}(s_1)$ for the set of blue crossings $s_1$ participates in and do the same with $s'_1$. Furthermore, we denote $\mathcal{B}_i(s'_1) := \mathcal{B}_i \cap \mathcal{B}(s'_1)$ for $i=0,1,2$. If a blue crossing is associated to no red non-crossing, then $s'_1$ has to participate in the crossing, i.e., $\mathcal{B}_0=\mathcal{B}_0(s'_1)$. Moreover, a blue crossing is associated to exactly one red non-crossing if either $s'_1$ is involved and we associate exactly one non-crossing or $s_1$ is involved in the crossing but not $s'_1$. This is, $\mathcal{B}_1=(\mathcal{B}(s_1)\backslash\mathcal{B}(s'_1))\cup \mathcal{B}_1(s'_1)$. Note that every red non-crossing is associated with at most two blue crossings and hence we have $|\mathcal{B}_1|+2|\mathcal{B}_2|\leq 2|\overline{\mathcal{R}}|$. This leads to:

 \vspace{-0.8em}
 \begin{displaymath}
  2|\mathcal{B}|\leq 2|\overline{\mathcal{R}}| + |\mathcal{B}(s_1)\backslash\mathcal{B}(s'_1)| + |\mathcal{B}_1(s'_1)| + 2|\mathcal{B}_0(s'_1)|
 \end{displaymath}

 First, note that Since $P, P'$ have only $j$ blue segments we clearly have $|\mathcal{B}(s_1)\backslash\mathcal{B}(s'_1)|\leq j-1$.

 Second, observe the following: On the path $P$ between any two blue segments contributing to a $\mathcal{B}_0(s'_1)$-crossing, either there is another blue segment contributing to $\mathcal{B}_0(s'_1)$ or there is a blue segment of $P$ that participates in a $\mathcal{B}_2(s'_1)$-crossing or there is one which does not cross $s'_1$ at all. Hence, because $P$ has only $j$ blue segments, we have $2|\mathcal{B}_0(s'_1)|-1+|\mathcal{B}_1(s'_1)|\leq j$.

 Plugging both into the above inequality, we calculate $2|\mathcal{B}|\leq 2|\overline{\mathcal{R}}|+2m$. Thus, $|\mathcal{B}|-|\overline{\mathcal{R}}|\leq j$. Now adding $j^2$ on both sides we obtain: $|\mathcal{B}|+|\mathcal{R}|\leq j^2 + j$, where $\mathcal{R}$ denotes the set of red crossings. In particular, the number of crossings is at most $j(j+1)$.
\end{proof}

Note that pretzels, see Figure~\ref{fig:cross}, witness that the bound in the above lemma is tight. The following question arises:
\begin{quest}\label{quest:cross}
 What is the maximum number of crossings that two $2j$-bend paths can have?
\end{quest}

Next we combine Lemma~\ref{lem:c-inequality} and Lemma~\ref{lem:pnt} to achieve a second Lower-Bound-Lemma.

\begin{lem}[\textbf{Lower-Bound-Lemma~II}]\label{lem:lower-bound-2}
 For every $k$-bend representation of $K_{m,n}$ we have
 \begin{displaymath}
  n(2m-k-2) \leq m(m-1)\lceil\frac{k+1}{2}\rceil\lceil\frac{k+3}{2}\rceil + 2(k+1)m.
 \end{displaymath}
\end{lem}
\begin{proof}
 By Lemma~\ref{lem:pnt} two $k$-bend paths can cross in at most $\lceil\frac{k+1}{2}\rceil\lceil\frac{k+3}{2}\rceil$ points. (If $k$ is even we simply take the bound for two $(k+1)$-bend paths.) Hence $m$ $k$-bend paths cross in at most $\binom{m}{2}\lceil\frac{k+1}{2}\rceil\lceil\frac{k+3}{2}\rceil$ points. Plugging this into Lemma~\ref{lem:c-inequality} gives the claimed inequality.
\end{proof}

We think that the bound on the number of crossings between $m$ $k$-bend paths can be further improved. In particular we conjecture that for odd $k$ no three $k$-bend paths can pairwise cross in $\frac{k+1}{2}\frac{k+3}{2}$ points. This would imply that for odd $k$ the blown-up pretzel maximizes the total number of crossings between $m$ $k$-bend paths.

\begin{thm}\label{thm:kmm4}
 We have $b(K_{m,n})=2m-2$ for all $n>m^4-2m^3+5m^2-4m$. Note that this leaves only a quadratic discrepancy to the bound in Theorem~\ref{thm:m4}.
\end{thm}
\begin{proof}
 If $m=1$, then $K_{m,n}$ is a star and thus an interval graph, i.e., $b(K_{1,n}) = 0$ for all $n > 0$.

 For $m > 1$ assume that $b(K_{m,n}) \leq 2m-3$. Then by the Lower-Bound-Lemma~II with $k = 2m-3$ we get $n \leq 2\binom{m}{2}(m-1)m + 2(2m-2)m = m^4 -2m^3 +5m^2-4m$.

 On the other hand it is known~\cite{Gol-09} and illustrated in Figure~\ref{fig:kmN} that $b(K_{m,n}) \leq 2m-2$, regardless of $n$.
\end{proof}

\begin{figure}[htb]
 \centering
 \includegraphics{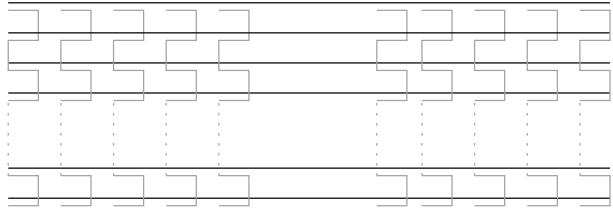}
 \caption{A representation verifying $b(K_{m,n})\leq 2m-2$.}
 \label{fig:kmN}
\end{figure}

%%%%%%%%%%%%%%%%%%%%%%%%%%%%%%%%%%%%%
%% N E W   V E R S I O N   E N D S %%
%%%%%%%%%%%%%%%%%%%%%%%%%%%%%%%%%%%%%

Applying all the above machinery together with the $2$-bend representation of $K_{3,10}$ in Figure~\ref{fig:k35} we obtain that $b(K_{3,n})$ is $1$ if $n\leq 2$, $2$ if $3\leq n\leq 10$, $3$ if $11\leq n\leq 39$ and $4$ if $n\geq 61$. The first unknown case is $3\leq b(K_{3,40})\leq 4$.

\section{Degeneracy}\label{sec:acy}

In~\cite{Bie-10} an acyclic orientation of $G$ with maximum indegree $k$ is referred to as a \emph{$k$-regular acyclic orientation}. The \emph{degeneracy} $\dg(G)$ of $G$ is the smallest number $k$, such that $G$ has a $k$-regular acyclic orientation, see~\cite{Lic-70}. In this section we provide a tight upper bound on the bend-number of graphs with a fixed degeneracy. In~\cite{Bie-10} the following result was suspected to be true.

% Recall that a set $S$ of subsegments of the same direction, say all vertical, sees each other if there is a (horizontal) grid-line that intersects each subsegment in a right angle -- Indeed we can assume w.l.o.g. that there is as many such grid-lines as we want. We say that a grid path $P$ \emph{connects $S$} if it is a ``snake'' path edge-intersecting all subsegments in $S$. More formally, there are two horizontal grid-lines $\ell_1$ and $\ell_2$ that intersect all subsegments in $S$ and contain all horizontal segments of $P$, while each vertical segment of $P$ is completely contained in one subsegment from $S$. A grid path that connects $S$ has exactly $2|S|-2$ bends.

\begin{thm}\label{thm:acy}
 Every graph $G$ has a $(2\dg(G)-1)$-bend representation.
\end{thm}
\begin{proof}
 Take a $\dg(G)$-regular acyclic orientation of $G$ in the above sense. A topological ordering then gives us a building recipe for $G$, where every new vertex will be connected to at most $\dg(G)$ vertices of the already constructed part. We construct a $(2\dg(G)-1)$-bend representation simultaneously to the building process of $G$. We maintain the following invariant for the $(2\dg(G)-1)$-bend representation of the already constructed subgraph $G'$ of $G$: For every vertex $v\in V(G')$ the path $P(v)$ contains a vertical subsegment $s(v)$ and a horizontal subsegment $\bar{s}(v)$ both displaying $v$. Moreover all the vertical subsegments $s(v)$ for $v\in V(G')$ see each other. In the first step, when $G'$ is just a single vertex $v$, this invariant holds by taking any $1$-bend path for $v$.

 Now when the next vertex $v$ is added to $G'$, consider the $k$ neighbors of $v$ in $G'$. In particular $k \leq \dg(G)$. Let these neighbors be labeled $v_1,\ldots,v_k$. We define the grid path $P(v)$ for $v$ as follows: We start with $P(v)$ being any grid path connecting $s(v_1),\ldots,s(v_{k-1})$ and extend the last vertical segment of $P(v)$ by one grid-edge, so that its end-point has a different $y$-coordinate than its bends. (If $k = 1$, we define $P(v)$ to be a single grid-point on a grid-line that intersects all the $s(v_i)$.) Next, we extend $P(v)$ at its end-point by three further segments so that the last horizontal segment is completely contained in $\bar{s}(v_k)$. We refer to Figure~\ref{fig:col} for an example of this construction. (If $k=0$, we define $P$ to be any $1$-bend path whose vertical segment sees all the $s(v_i)$.)

 \begin{figure}[htb]
  \psfrag{v}[cc][bl]{$s(v)$}
  \psfrag{v'}[cc][bl]{$\bar{s}(v)$}
  \psfrag{1}[cc][bl]{$s(v_1)$}
  \psfrag{2}[cc][bl]{$s(v_2)$}
  \psfrag{3}[cc][bl]{$s(v_3)$}
  \psfrag{4}[cc][bl]{$s(v_4)$}
  \psfrag{4'}[cc][bl]{$\bar{s}(v_4)$}
  \centering
  \includegraphics{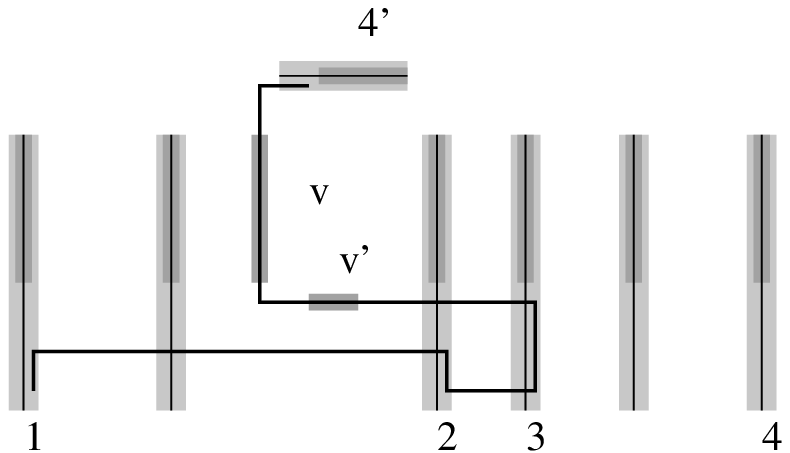}
  \caption{Building a $(2\dg(G)-1)$-bend representation of $G$, a path $P(v)$ is inserted. Vertex-displaying subsegments before and after the insertion are highlighted with light-grey and dark-grey, respectively.}
  \label{fig:col}
 \end{figure}

 The path $P(v)$ has $2k-2 + 3 \leq 2\dg(G)-1$ bends and is edge-intersecting exactly the paths corresponding to $v_1,\ldots,v_k$. Moreover, the last vertical segment $s$ of $P(v)$ displays $v$ and sees all previous vertical subsegments. Thus for each vertex $w$ in $G' \cup v$ we can find a vertical subsegment $s(w)$ as required in our invariant. Finally, a horizontal subsegment for $v_k$ and $v$ can be chosen as a subset of the previous $\bar{s}(v_k)$ and any horizontal segment of $P(v)$, respectively. Thus our invariant is maintained for the new subgraph $G' \cup v$, which proves the theorem.
\end{proof}

Next, we show that Theorem~\ref{thm:acy} is worst-case optimal, even for bipartite graphs.

\begin{thm}
 For every $m$ there is a bipartite $G$ with $\dg(G)\leq m$ and $b(G)\geq 2\dg(G)-1$.
\end{thm}
\begin{proof}
 The graph $G$ arises from a $K_{m,n}$ with large enough $n = |B|$, which will be determined later. First, for every $m$-subset $B'$ of $B$ we add $m(2m-2)+1$ new vertices, each having $B'$ as its neighborhood. The set of the $(m(2m-2)+1)\binom{n}{m}$ added vertices is denoted by $C$. Moreover, for every $m$-subset $C'$ of $C$ we add $m(2m-2)+1$ new vertices, each having $C'$ as its neighborhood. This set of vertices is denoted by $D$.

 Clearly, $\dg(G) \leq m$. Now suppose $b(G) \leq 2m-2$ and consider the $K_{m,n}$ induced by $A$ and $B$. Indeed by Lemma~\ref{lem:pnt} there are at most $\binom{m}{2}m(m+1)$ crossings between the $m$ $(2m-2)$-bend paths for $A$. Every crossing can lie on at most two paths for $B$. Moreover, the paths for $A$ have only $m(4m-4)$ endpoints and every such endpoint lies on at most one path for $B$. It follows that for at least $n-4m^4$ vertices $b \in B$ the $m$ edge-intersections of $P(b)$ with paths for $A$ appear on exactly every second segment of $P(b)$. By this with increasing $n$ the paths for an arbitrarily large subset $\widetilde{B} \subset B$ must look like in Figure~\ref{fig:kmN}, i.e., all edge-intersections of paths for $\widetilde{B}$ with paths for $C$ have the same orientation, say vertical. Additionally, there are many pairs $B_1,B_2$ of $m$-subsets of $\widetilde{B}$ with the property that every path for $B_1$ lies to the left of every path for $B_2$.

 We fix $m$ distinct $m$-subsets $B_1,\ldots,B_m$ of $\widetilde{B}$, such that every path for $B_1$ lies to the left of every path for $B_2 \cup \ldots \cup B_m$. Hence the path for every vertex $c \in C$ whose neighborhood is $B_1$ lies completely to the left of every path for a vertex in $C$ whose neighborhood is one of $B_2, \ldots, B_m$.

 \begin{figure}[htb]
  \psfrag{c}[cc][cc]{$P(c_1)$}
  \psfrag{c2}[cc][cc]{$P(c_2)$}
  \psfrag{c3}[cc][cc]{$P(c_3)$}
  \centering
  \includegraphics{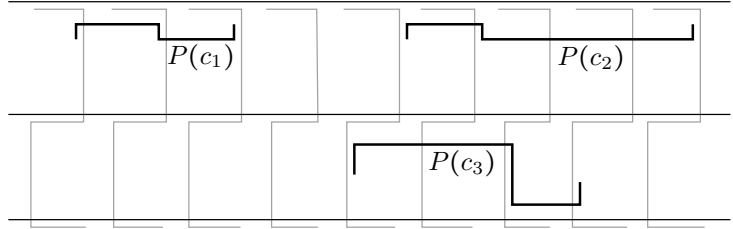}
  \caption{A part of a hypothetical $(2\dg(G)-2)$-bend representation of $G$. The paths for the sets $\widetilde{B}$ and $\widetilde{C} = \{c_1,c_2,c_3\}$ are depicted grey and thick, respectively.}
  \label{fig:col-tight}
 \end{figure}

 Now for every $B_i$ there are $m(2m-2)+1$ vertices in $C$ whose neighborhood is $B_i$, that is, one more than there are bends in the paths for $B_i$. Since every bend of a path for $B_i$ lies on at most one path for $C$, there is for every $i\in\{1,\ldots,m\}$ at least one vertex $c_i \in C$ whose neighborhood is $B_i$ and for which the vertical segments of $P(c_i)$ are completely contained in segments of path for $B_i$ (see Figure~\ref{fig:col-tight} for an example). In $G$ there is a set $D'\subset D$ of $m(2m-2)+1$ vertices with neighborhood $C' := \{c_1,\ldots,c_m\}$; again one more than there are bends in $C'$. Hence of at least one path for a vertex in $D'$ all \emph{horizontal} segments are contained in segments of paths for $C'$. But this is impossible since the path $P(c_1)$ lies to the left of the paths for all other vertices in $C'$.
\end{proof}

In~\cite{Bie-10} a (not necessarily acyclic) orientation of $G$ with maximum indegree $k$ is referred to as a \emph{$k$-regular orientation}. The \emph{pseudo-arboricity} $\pa(G)$ of $G$ is the smallest number $k$, such that $G$ has a $k$-regular orientation, see~\cite{Pic-82}. In~\cite{Bie-10} it is shown that $b(G)\leq 2\pa(G)+1$. 

\begin{quest}\label{quest:pa}
 What is the maximal bend-number of graphs with bounded pseudo-arboricity?
\end{quest}

%%%%%%%%%%%%%%%%%%%%%%%%%%%%%%%%%%%%%%%%%%%%%%%%%%%%%%%%%%%%%%%%%%%%%%%%%%%%%%%%%%%%%%%%%%%%%%%%%%%%%
%%%%%%%%%%%%%%%%%%%%%%%%%%%%%%%%%%%%%%%%%%%%%%%%%%%%%%%%%%%%%%%%%%%%%%%%%%%%%%%%%%%%%%%%%%%%%%%%%%%%%
%%%%%%%%%%%%%%%%%%%%%%%%%%%%%%%%%%%%%%%%%%%%%%%%%%%%%%%%%%%%%%%%%%%%%%%%%%%%%%%%%%%%%%%%%%%%%%%%%%%%%
%%%%%%%%%%%%%%%%%%%%%%%%%%%%%%%%%%%%%%%%%%%%%%%%%%%%%%%%%%%%%%%%%%%%%%%%%%%%%%%%%%%%%%%%%%%%%%%%%%%%%

\section{Treewidth}\label{sec:tw}
According to~\cite{Pat-86} a \emph{$k$-tree} is a graph that can be constructed starting with a $(k+1)$-clique and in every step attaching a new vertex to a $k$-clique of the already constructed graph. The \emph{treewidth} $\tw(G)$ of a graph $G$ is the minimum $k$ such that $G$ is a subgraph of some $k$-tree~\cite{Rob-85}.

In~\cite{Gol-09} Golumbic, Lipshteyn, and Stern show, that every graph $G$ with $\tw(G) \leq 1$, i.e., every forest, has $b(G) \leq 1$.  In~\cite{Hel-12} it is shown that graphs of treewidth at most $2$ have bend-number at most $2$. In this section we present a tight bound for the maximum bend-number of graphs of bounded treewidth.

To judge the upper bounds of the bend-number in terms of the bounded treewidth, we remark, that there is a graph $G$ with $\tw(G) = 1 = b(G)$, i.e. the bound for the class of graphs of treewidth 1 is tight. The general lower bound is a simple corollary of a previously mentioned theorem.
\begin{cor}\label{prop:bend-tw-low}
 For every $k\geq 2$ there is a graph $G$ with $\tw(G) = k$ and $b(G) = 2\cdot \tw(G) - 2$.
\end{cor}
\begin{proof}
 Theorem~\ref{thm:kmm4} states, that the complete graph $K_{m,n}$ with $n\gg m$  has bend-number $2m-2$. Clearly this graph has treewidth $m$.
\end{proof}
%So the bounds, we have seen so far, match and have to be tight. 
%We conclude this section with an EPG-construction for graphs $G$ with $\tw(G) > 2$:
%%%%%%%%%%%%%%%%%%%%%%%%%%%%%%%%%%%%%%%%%%%%%%
% b(G) \leq 2 tw(G) -2 

\begin{thm}\label{thm:twUpperBound}
 Let $k\geq 2$. For every graph $G$ with $\tw(G) \leq k$ we have $b(G) \leq 2k - 2$.
\end{thm}
\begin{proof}
 The cases $\tw(G)\leq 1,2$ are solved by~\cite{Gol-09,Hel-12}. So let $\tw(G)=k\geq3$.
 Let $\tilde{G}$ be a $k$-tree that is a super-graph of $G$. Then $\tilde{G}$ can be iteratively constructed starting with a $k$-clique and adding one-by-one the remaining vertices to the graph, such that the neighbors of each vertex in the already constructed graph form a $k$-clique in $\tilde{G}$. Similarly to the proof of Theorem~\ref{thm:acy} we construct a $(2k-2)$-bend representation of $G$ along this building sequence of $G$. Again when a new vertex $v$ is added to the already constructed subgraph $G'$ of $G$, then $v$ has at most $k$ neighbors in $G'$. But this time only few $k$-sets of vertices can form the neighborhood of a later added vertex and we know exactly which -- namely the $k$-cliques in $\tilde{G}$.

 We maintain an invariant on the $(2k-2)$-bend representation of the graph $G'$. Loosely speaking, we ensure that for ``almost every'' vertex $w$ we have a vertical and a horizontal subsegment that displays $w$ and crosses the $x$-axis and $y$-axis, respectively. In particular all these vertical subsegments, as well as all the horizontal ones, see each other. More formally, we fix a coordinate cross with origin $o$, horizontal $x$-axis and vertical $y$-axis. We say that a vertex $w$ is \emph{displayed horizontally (vertically)} if there is a horizontal (vertical) subsegment of $P(v)$ that displays $v$ and intersects the $y$-axis ($x$-axis) in an interior point. We require that for every $k$-set $W$ of vertices in $G'$ that forms a $k$-clique in $\tilde{G}$ one of the following invariants holds.

 \begin{description}
  \item[Invariant A:] There exists a vertex $w_1 \in W$ such that every $w \in W \setminus w_1$ is displayed horizontally, and every $w \in W$ is displayed vertically. Moreover, on either side of $o$ there is at least one of the vertical subsegments. This invariant is also satisfied if it holds exchanging horizontal and vertical. It is illustrated in Figure~\ref{fig:bend-tw-up-A}.
  \item[Invariant B:] There exist two adjacent vertices $w_1,w_2 \in W$ such that every $w \in W \setminus w_1$ is displayed horizontally, and every $w \in W \setminus w_2$ is displayed vertically. Moreover, on the same grid-line as the subsegment for $w_1$ lies a subsegment $s(w_1w_2)$ that displays the edge $(w_1,w_2)$.  This invariant is also satisfied if it holds exchanging horizontal and vertical. It is illustrated in Figure~\ref{fig:bend-tw-up-B}.
 \end{description}

 \begin{figure}[htb]
  \centering
  \psfrag{w1}{$w_1$}
  \psfrag{w2}{$w_2$}
  \psfrag{w3}{$w_3$}
  \psfrag{v}[cc][cc]{$P(v)$}
  \includegraphics{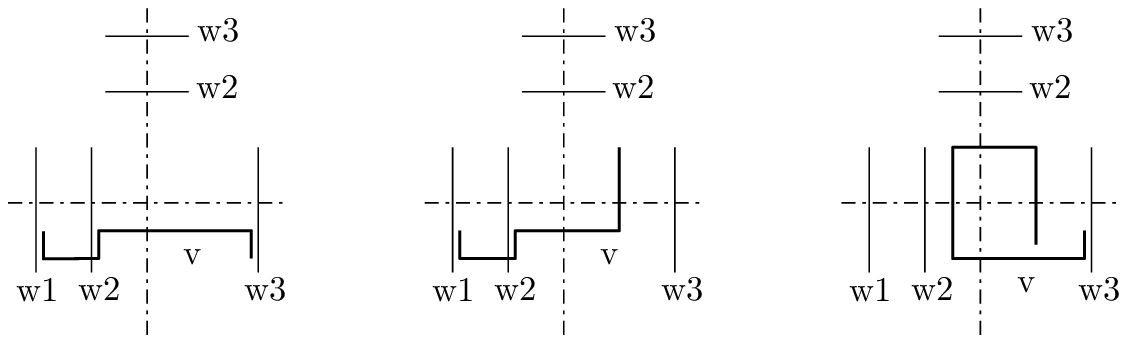}
  \caption{Invariant A for a $k$-clique $\{w_1,w_2,w_3\}$ of $\tilde{G}$ and how to insert the new path $P(v)$ (drawn bold). Left: $v$ is adjacent to every vertex in the $k$-clique. Middle: $v$ not adjacent to at least one vertex in the $k$-clique. Right: $v$ is adjacent only to the unique vertex with a vertical subsegment to the right of $o$.}
  \label{fig:bend-tw-up-A}
 \end{figure}

 When the next vertex $v$ is added to $G'$ in the building sequence of $G$ we define the corresponding grid path $P(v)$ as follows. Let $W$ be the $k$-clique of $\tilde{G}$ that contains the neighborhood of $v$ in $G'$. First suppose that invariant A holds for $W$. We assume w.l.o.g. that every $w \in W$ is displayed vertically and every $w \in W \setminus w_1$ is displayed horizontally. Let $S$ be the set of vertical subsegments that correspond to the neighbors of $v$ in $G'$, display the corresponding vertex, and intersect the $x$-axis. We define $P(v)$ to be a grid path that connects $S$ and lies completely on one side of the $x$-axis. If $|S| = k$, i.e., $v$ is adjacent to every vertex in $W$, then $P(v)$ is a $(2k-2)$-bend path. If $|S| \leq k-1$, we extend $P(v)$ by two segments, such that either segment intersects a coordinate axis. If the vertical subsegments of all but one vertex $w^* \in W$ lie on the same side of $o$ and $v$ is adjacent only to $w^*$, i.e., $|S|=1$, then we extend $P(v)$ by two further segments that intersect the coordinate axes. In particular $P(v)$ is extended by four segments in total. This case is exemplified in the right of Figure~\ref{fig:bend-tw-up-A}. (Here we make implicit use of the assumption $k \geq 3$ to ensure that $P(v)$ has at most $2k-2$ bends.)

 We have obtained a $(2k-2)$-bend representation of $G' \cup v$ and we claim that our invariant holds again. Indeed every vertex different from $v$ is displayed horizontally or vertically if and only if it was so before. Thus the invariant still holds for every $k$-clique not containing the new vertex $v$. Let $W'$ be a $k$-clique of $\tilde{G}$ that contains $v$. Every vertex in $W'\setminus v$ is displayed vertically, and every vertex in $W' \setminus w_1,v$ is displayed horizontally. If $|S|=k$ then $v$ is displayed horizontally since there is at least one subsegment from $S$ on the left and on the right of $o$. Moreover, there is a subsegment displaying every edge of the form $(v,w)$ with $w \in W$. In particular invariant B holds for $W'$.

 If $|S|\leq k-1$ then $v$ is displayed horizontally and vertically due to the additional segments of $P(v)$ that were added last. If $P(v)$ has been extended by four segments, then $P(v)$ is even displayed vertically on both sides of $o$. In both cases there is a vertical subsegment corresponding to a vertex in $W'$ on either side of $o$, which implies that invariant A holds for $W'$.

 \begin{figure}[htb]
  \centering
  \psfrag{w1}{$w_1$}
  \psfrag{w2}{$w_2$}
  \psfrag{w3}{$w_3$}
  \psfrag{s}[cc][cc]{$s(w_1w_2)$}
  \psfrag{v}{$P(v)$}
  \includegraphics{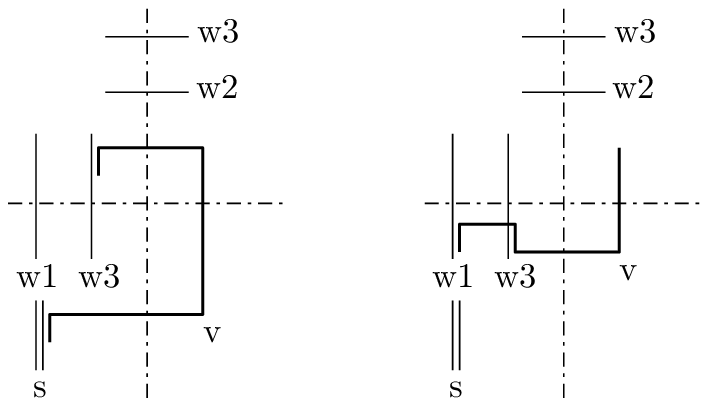}
  \caption{Invariant B for a $k$-clique $\{w_1,w_2,w_3\}$ of $\tilde{G}$ and how to insert the new path $P(v)$ (drawn bold). Left: $v$ is adjacent to $w_1$ and $w_2$. Right: $v$ not adjacent to $w_2$.}
  \label{fig:bend-tw-up-B}
 \end{figure}

 Next, suppose that invariant B holds for $W$. First we consider the case that $v$ is adjacent to $w_1$ and $w_2$. Due to symmetry we may assume that every $w \in W \setminus w_1$ is displayed horizontally, every $w \in W \setminus w_2$ is displayed vertically, and that $s(w_1w_2)$ is vertical and lies below the $x$-axis. Let $S$ be the set of those vertical subsegments that correspond to the neighbors of $v$ different from $w_1,w_2$. Let $P(v)$ be a grid path that connects $S$ and lies completely above the $x$-axis. We define a short vertical segment $s$ that is completely contained in $s(w_1w_2)$. We connect $P(v)$ and $s$ --making $s$ the last segment of $P(v)$-- by inserting two horizontal and one vertical segment in between. We choose these three new segments such that the new vertical segment and at least one of the new horizontal segments intersect the $x$-axis and $y$-axis, respectively. Moreover, if all vertical or all horizontal subsegments for $W \setminus w_2$ or $W\setminus w_1$ lie on the same side of $o$, respectively, then we choose the new segments such that $P(v)$ intersects the corresponding axis on the opposite side.

 Secondly, if $v$ not adjacent to $w_1$ or $w_2$ -- say $w_2$, then we define $P(v)$ as we did for invariant A in case $|S| \leq k-1$. In this case we do not need the special treatment if $|S|=1$.
 
 In both cases, the result is a $(2k-2)$-bend representation of $G' \cup v$ and we claim that our invariant still holds. By definition $v$ is vertically as well as horizontally displayed. As before, let $W'$ be a $k$-clique of $\tilde{G}$ consisting of $v$ and all but one vertex from $W$. If $w_2 \notin W'$ then invariant A holds. If $w_1 \notin W'$ then invariant A holds by interchanging the roles of vertical and horizontal. If $w_1,w_2 \in W'$ then invariant B holds since a subset of the former $s(w_1w_2)$ still displays the edge $(w_1,w_2)$.
\end{proof}

By Corollary~\ref{prop:bend-tw-low} and Theorem~\ref{thm:twUpperBound} the maximum bend-number for the class of graphs with a fixed treewidth is completely determined. In~\cite{Kna-12} the \emph{simple treewidth} of a graph is defined as follows. A \emph{simple $k$-tree} is a graph that can be constructed starting with a $(k+1)$-clique and in every step attaching a new vertex to a $k$-clique of the already constructed graph such that at most one vertex is attached to the same $k$-clique. The \emph{simple treewidth} $\stw(G)$ of a graph $G$ is the minimum $k$ such that $G$ is a subgraph of some simple $k$-tree. This is a graph-parameter of independent interest, see~\cite{Kna-12b}. In~\cite{Hel-12} it is shown that $\stw(G)\leq 3$ implies $b(G)\leq 3$, being one less than what is needed for $\tw(G)\leq 3$. 

\begin{conj}\label{conj:stw}
 For $k\geq 3$ the maximal bend-number of graphs with $\stw(G)\leq k$ is $2k-3$.
\end{conj}

\section{Complexity}\label{sec:rec}
In~\cite{Asi-12} it is asked for the complexity of recognizing $k$-bend graphs. In general, the bend-number of a graph can be computed by solving a mixed integer program (MIP). Unfortunately the problem instance becomes so huge, that this approach is inapplicable even for graphs with only 10 vertices. It is well-known that interval graphs, that is $0$-bend graphs, can be recognized in polynomial time~\cite{Boo-76}. In this section we prove that recognizing single-bend graphs ($1$-bend graphs) is NP-complete. In~\cite{Gya-95} it was shown that recognizing 2-track graphs is NP-complete and~\cite{Shm-84} proves that recognizing $k$-interval graphs is NP-complete for every fixed $k\geq 2$. Finally, in~\cite{Jia-10} it is shown that $k$-track graphs is NP-complete for every fixed $k\geq 2$. One easily sees that every single-bend graph is a 2-track graph as well as a 2-interval graph. But the converse is not true. For example every outerplanar $G$ has $t(G) \leq 2$~\cite{Kos-99} and $i(G)\leq 2$~\cite{Sch-83}.%, but is not necessarily a single-bend graph (see Figure~\ref{fig:out-tight}).

It is easy to verify a single-bend representation, so SINGLE-BEND-RECOG\-NITION is in NP. For NP-hardness we set up a reduction from ONE-IN-THREE 3-SAT, i.e., we are given a formula $\mathcal{F}=(\mathcal{C}_{1} \wedge \cdots \wedge \mathcal{C}_{n})$ that is a conjunction of clauses $\mathcal{C}_{1},\ldots,\mathcal{C}_{n}$. Each clause is the \emph{exclusive disjunction} of exactly three literals $\mathcal{C}_{i} = (x_{i1} \veedot x_{i2} \veedot x_{i3})$ which are in turn either negated or non-negated Boolean variables. Given such a formula $\mathcal{F}$, it is NP-complete \cite{Gar-79,Sch-78} to decide, whether there is an assignment of the variables fulfilling $\mathcal{F}$, that is in each clause there is \emph{exactly one true literal}. Moreover ONE-IN-THREE 3-SAT remains NP-complete if each literal is a \emph{non-negated variable} and each clause consists of three \emph{distinct literals}. We will use both additional assumptions on $\mathcal{F}$, even though the first is just for convenience. The distinctness assumption is crucial in the following reduction.

Given a ONE-IN-THREE 3-SAT formula $\mathcal{F}$ we will define a graph $G_\mathcal{F}$, such that $b(G_\mathcal{F}) = 1$ if and only if $\mathcal{F}$ can be fulfilled. The graph consists of an induced subgraph $G_\mathcal{C}$ for every clause $\mathcal{C}$ with 13 vertices, called the \emph{clause gadget}, a vertex $v_{j}$ for every variable $x_{j}$ and 31 additional vertices.

\subsection{Clause Gadgets}
Constructing a clause gadget $G_\mathcal{C}$ starts with an induced octahedral graph $O$. Label the vertices by $\{a,A,b,B,c,C\}$ as in Figure~\ref{fig:octahedral-graph}. This way $\{a,A\}$, $\{b,B\}$ and $\{c,C\}$ are the three \emph{non}-edges and their complements $\{b,C,B,c\}$, $\{a,C,A,c\}$ and $\{a,B,A,b\}$ are the three induced 4-cycles in $O$.

\begin{figure}[htb]
 \centering
 \psfrag{a}[cc][cc]{$a$}
 \psfrag{b}[cc][cc]{$b$}
 \psfrag{c}[cc][cc]{$c$}
 \psfrag{A}[cc][cc]{$A$}
 \psfrag{B}[cc][cc]{$B$}
 \psfrag{C}[cc][cc]{$C$}
 \psfrag{Pa}[cc][cc]{$P(a)$}
 \psfrag{Pb}[cc][cc]{$P(b)$}
 \psfrag{Pc}[cc][cc]{$P(c)$}
 \psfrag{PA}[cc][cc]{$P(A)$}
 \psfrag{PB}[cc][cc]{$P(B)$}
 \psfrag{PC}[cc][cc]{$P(C)$}
 \psfrag{1}[cc][cc]{\textbf{1)}}
 \psfrag{2}[cc][cc]{\textbf{2)}}
 \includegraphics{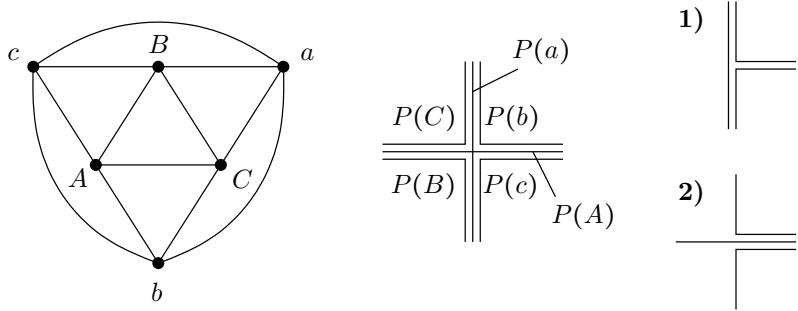}
 \caption{The labeled octahedral graph $O$, a single-bend representation of $O$, and the two possible ways a triangle of $O$ is represented.}
 \label{fig:octahedral-graph}
\end{figure}

\begin{lem}\label{lem:octaeder}
 We have $b(O) = 1$ and in every single-bend representation
 \begin{enumerate}[label=(\roman{*}), ref=(\roman{*})]
  \item there is a unique grid-point, called the \emph{center}, that is contained in every path,\label{enum:center}
  \item every edge-intersection between two paths lies on a half ray starting at the center, called a \emph{center ray},
  \item for every pair of center rays, there is a unique vertex in $O$ whose path intersects exactly these two center rays, and\label{enum:pairs}
  \item every triangle in $O$ is represented in one of the two ways on the right of Figure~\ref{fig:octahedral-graph}.\label{enum:triangle}
 \end{enumerate}
\end{lem}

\begin{proof}
 Figure~\ref{fig:octahedral-graph} shows $b(O)\leq 1$ and since $O$ contains induced 4-cycles it is not an interval graph. Hence $b(O) =1$.

 By a result of~\cite{Gol-09}, every induced 4-cycle in a single-bend representation is either a \emph{frame}, a \emph{true pie} or a \emph{false pie}. These terms are illustrated in Figure~\ref{fig:4-cycles}. If an induced 4-cycle is represented by a frame, then the bends of the four corresponding paths are pairwise distinct. Thus in a single-bend representation no other single-bend path can edge-intersect all of them. Since for each induced 4-cycle in $O$ there is a vertex that is adjacent to all of its vertices, we conclude that $\{a,B,A,b\}$, $\{a,C,A,c\}$ and $\{b,C,B,c\}$ are pies. So all pies share the middle point, the claimed center, and every path intersects exactly two center rays. Since every edge in $O$ is part of an induced 4-cycle, no two paths can intersect the same pair of center rays. This concludes~\ref{enum:center}--\ref{enum:pairs}. Part~\ref{enum:triangle} is easily obtained from this.
\end{proof}

\begin{figure}[htb]
 \centering
 \includegraphics{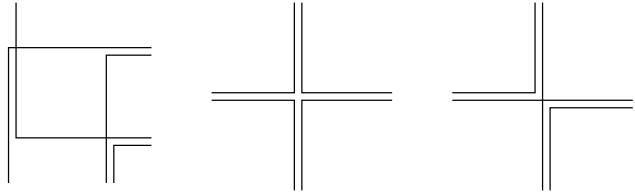}
 \caption{Single-bend representations of an induced 4-cycle: A frame (left), a true pie (middle) and a false pie (right).}
 \label{fig:4-cycles}
\end{figure}

To complete a clause gadget $G_\mathcal{C}$ seven vertices are added to the octahedral graph $O$: $W_{ABC}$ is adjacent to $\{A,B,C\}$, $w_{abC}$, $w_{aBc}$ and $w_{Abc}$ are adjacent to $\{a,b,C\}$, $\{a,B,c\}$ and $\{A,b,c\}$, respectively, and $s_{ab}$, $s_{ac}$ and $s_{bc}$ are adjacent to $\{a,b\}$, $\{a,c\}$ and $\{b,c\}$, respectively. The resulting graph is depicted in Figure~\ref{fig:clause-gadget}.

\begin{figure}[htb]
 \centering
 \psfrag{a}[cc][cc]{$a$}
 \psfrag{b}[cc][cc]{$b$}
 \psfrag{c}[cc][cc]{$c$}
 \psfrag{A}[cc][cc]{$A$}
 \psfrag{B}[cc][cc]{$B$}
 \psfrag{C}[cc][cc]{$C$}
 \psfrag{ab}[cc][cc]{$s_{ab}$}
 \psfrag{bc}[cc][cc]{$s_{bc}$}
 \psfrag{ac}[cc][cc]{$s_{ac}$}
 \psfrag{W}[cc][cc]{$W_{ABC}$}
 \psfrag{w1}[cc][cc]{$w_{Abc}$}
 \psfrag{w2}[cc][cc]{$w_{aBc}$}
 \psfrag{w3}[cc][cc]{$w_{abC}$}
 \psfrag{Pab}[cc][cc]{$P(s_{ab})$}
 \psfrag{Pbc}[cc][cc]{$P(s_{bc})$}
 \psfrag{Pac}[cc][cc]{$P(s_{ac})$}
 \psfrag{PW}[cc][cc]{$P(W_{ABC})$}
 \psfrag{Pw1}[cc][cc]{$P(w_{Abc})$}
 \psfrag{Pw2}[cc][cc]{$P(w_{aBc})$}
 \psfrag{Pw3}[cc][cc]{$P(w_{abC})$}
 \includegraphics{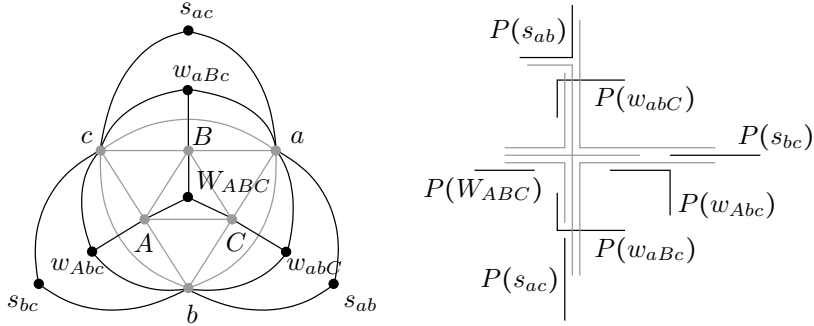}
 \caption{The clause gadget $G_\mathcal{C}$ with a single-bend representation.}
 \label{fig:clause-gadget}
\end{figure}

\begin{lem}\label{lem:clause-invariant}
 We have $b(G_\mathcal{C})= 1$ and in every single-bend representation
 \begin{enumerate}[label=(\roman{*}), ref=(\roman{*})]
  \item every center ray contains a segment of exactly one of $P(W_{ABC})$, $P(w_{abC})$, $P(w_{aBc})$, and $P(w_{Abc})$, and\label{enum:w}
  \item every such segment, except the one of $P(W_{ABC})$, is contained in a segment of $P(a)$, $P(b)$, or $P(c)$.\label{enum:w-s}
 \end{enumerate}
\end{lem}

\begin{proof}
 Let $w \in \{W_{ABC},w_{abC},w_{aBc},w_{Abc}\}$. Then $w$ is adjacent to every vertex of the triangle $\Delta$ in $O$ that is induced by the vertices in the subscript of $w$. By Lemma~\ref{lem:octaeder}--\ref{enum:triangle}, $\Delta$ is represented in one of the two ways that are illustrated on the right of Figure~\ref{fig:octahedral-graph}. In case \textbf{1)}, $P(w)$ would be contained in two center rays since it has an edge-intersection with all three paths. But then, by Lemma~\ref{lem:octaeder}--\ref{enum:pairs}, $P(w)$ would edge-intersect paths that correspond to vertices which are not adjacent to $w$. Hence $\Delta$ is represented as in case \textbf{2)} and one segment of $P(w)$ is contained in the center ray that supports the paths for all three vertices in $\Delta$. This concludes part \ref{enum:w}.

 Now consider a pair $(w,s)$ in $\{(w_{abC},s_{ab}), (w_{aBc},s_{ac}), (w_{Abc},s_{bc})\}$. Both, $P(w)$ and $P(s)$, intersect at most one center ray. Moreover it is the same center ray and it contains a segment of the path for the capitalized vertex that is adjacent to $w$ but not $s$. Hence the segment of $P(s)$ lies further away from the center than the segment of $P(w)$. Thus the segment of $P(w)$ is completely contained in a segment of the path for each neighbor of $s$.
\end{proof}

\subsection{The reduction}

Given a formula $\mathcal{F} = (\mathcal{C}_{1} \wedge \cdots \wedge \mathcal{C}_{n})$ with clauses $\mathcal{C}_{i} = (x_{i1} \veedot x_{i2} \veedot x_{i3})$ for $i =1,\ldots, n$ we are now ready to define the graph $G_\mathcal{F}$ as follows. See Figure~\ref{fig:3sat-example} for an example.
\begin{enumerate}
 \item For each clause $\mathcal{C}$ there is a clause gadget $G_\mathcal{C}$.
 \item For each variable $x_{j}$ there is a vertex $v_{j}$ that is adjacent to $w_{Abc}$, $w_{aBc}$, or $w_{abC}$, whenever $x_{j}$ is the first, second, or third variable in $\mathcal{C}$, respectively.
 \item There is a vertex $V$ adjacent to every $W$ in the clause gadgets.
 \item There is a $K_{2,4}$ with a specified vertex $T$ of the larger part, called the \emph{truth-vertex}. $T$ is adjacent to every $v_{j}$ and $V$.
 \item There are two octahedral graphs $O_1$ and $O_2$. The vertex $T$ is connected to the vertices of a triangle of each.
 \item There are two more octahedral graphs $O_3$ and $O_4$. The vertex $V$ is connected to the vertices of a triangle of each.
\end{enumerate}

\begin{figure}[htb]
 \centering
 \psfrag{V}[cc][cc]{$V$}
 \psfrag{T}[cc][cc]{$T$}
 \psfrag{x1}[cc][cc]{$v_1$}
 \psfrag{x2}[cc][cc]{$v_2$}
 \psfrag{x3}[cc][cc]{$v_3$}
 \psfrag{x4}[cc][cc]{$v_4$}
 \includegraphics{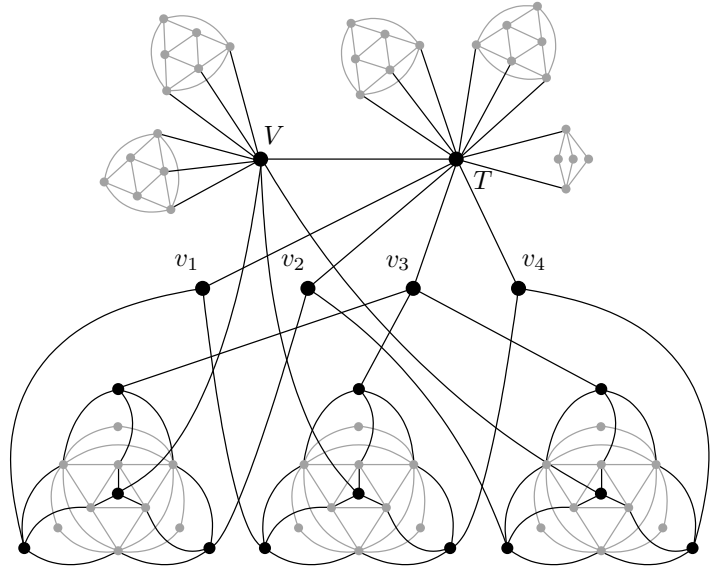}
 \caption{The graph $G_\mathcal{F}$ for $\mathcal{F} = (x_1\veedot x_2\veedot x_3) \wedge (x_1\veedot x_3 \veedot x_4) \wedge (x_2 \veedot x_3 \veedot x_4)$.}
 \label{fig:3sat-example}
\end{figure}

We will prove that a ONE-IN-THREE 3-SAT-formula $\mathcal{F}$ can be fulfilled if and only if $b(G_\mathcal{F})=1$.

\begin{thm}
SINGLE-BEND-RECOGNITION is NP-hard.
\end{thm}

\begin{proof}
First suppose $b(G_\mathcal{F}) = 1$ and consider a single-bend representation of $G_\mathcal{F}$. W.l.o.g. assume, that $P(V) \cap P(T)$ is a horizontal edge-intersection. We set a variable $x_j$ true if the edge-intersection between $P(v_j)$ and the path $P(T)$ for the truth-vertex is horizontal and false if the edge-intersection between $P(v_j)$ and $P(T)$ is vertical.

Note that in every single-bend representation of a $K_{2,4}$, the path for each vertex of the larger part, in particular $P(T)$ here, has its bend in a false pie (see~\cite{Asi-09} for a reasoning). The truth-vertex $T$ is adjacent to the vertices of a triangle of $O_1$ and $O_2$. From Lemma~\ref{lem:octaeder} follows that a segment of $P(T)$ is contained in exactly one center ray of each, $O_1$ and $O_2$. As the bend of $P(T)$ is in a false pie of $K_{2,4}$, the endpoints of $P(T)$ are contained in $O_1$ and $O_2$, respectively. Hence every further edge-intersection of $P(T)$ is completely contained in a segment of $P(T)$. Consequently, each path $P(v_j)$ edge-intersects the path $P(w)$ for $w\in\{w_{abC},w_{aBc},w_{Abc}\}$ in each clause gadget with its vertical segment if and only if $x_j$ is true.

For the same reason every edge-intersection of $P(V)$ other than that with $P(T)$ is vertical. Since $V$ is adjacent to the vertices of a triangle in $O_3$ and $O_4$, the two endpoints of the vertical segment of $P(V)$ are contained in $O_3$ and $O_4$, respectively. Thus, the vertical segment of the path $P(W_{ABC})$ for each clause gadget is completely contained in the vertical segment of $P(V)$. In consequence, the horizontal segment of every such $P(W_{ABC})$ by Lemma~\ref{lem:clause-invariant}--\ref{enum:w} is contained in a horizontal center ray. Hence of the other three center rays, two are vertical and one is horizontal. Together with \mbox{Lemma~\ref{lem:clause-invariant}--\ref{enum:w-s}} this yields, that in every clause gadget the edge-intersections between exactly two of $\{P(w_{abC}),P(w_{aBc}),P(w_{Abc})\}$ with the corresponding $P(v_j)$ is horizontal and exactly one is vertical. In other words every clause contains exactly one true variable.

\begin{figure}[htb]
 \centering
 \psfrag{O1}[cc][cc]{$O_3$}
 \psfrag{O2}[cc][cc]{$O_4$}
 \psfrag{O3}[cc][cc]{$O_1$}
 \psfrag{O4}[cc][cc]{$O_2$}
 \psfrag{K24}[cc][cc]{$K_{2,4}$}
 \psfrag{t}[cc][cc]{true}
 \psfrag{f}[cc][cc]{false}
 \psfrag{t1}[cc][cc]{true}
 \psfrag{f1}[cc][cc]{false}
 \psfrag{W}[cc][cc]{$P(W_{ABC})$}
 \psfrag{T}[cc][cc]{$P(T)$}
 \psfrag{V}[cc][cc]{$P(V)$}
 \includegraphics{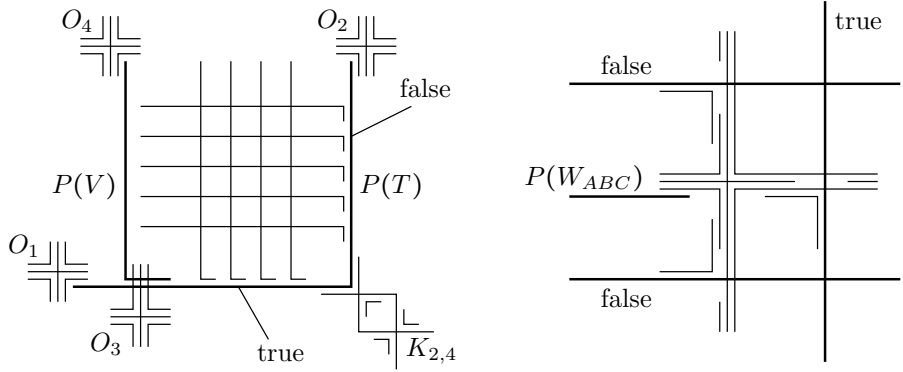}
 \caption{On the left: A single-bend representation of $G_\mathcal{F}$. The paths for the vertex $V$ and the truth-vertex $T$ are drawn bold. The clause gadgets are omitted. On the right: A single-bend representation of a clause gadget $G_\mathcal{C}$. The paths for vertices that correspond to the variables in the clause $\mathcal{C}$ and the vertex $W_{ABC}$ of the clause gadget are drawn bold.}
 \label{fig:3satrepresentation}
\end{figure}

\medskip

\noindent Now given a truth assignment fulfilling $\mathcal{F}$, we can construct a single-bend representation of $G_\mathcal{F}$. First represent all of $G_\mathcal{F}$ but the clause gadgets as on the left side in Figure~\ref{fig:3satrepresentation}. A path $P(v_j)$ is connected to the path for the truth-vertex $T$ horizontally if $x_j$ is true and vertically if $x_j$ is false.

To interlace a clause gadget $G_\mathcal{C}$, introduce a horizontal grid-line $l_h$ between the horizontal grid-lines used by the paths for the two false variables in $\mathcal{C}$. Then connect the path $P(W_{ABC})$ in $G_\mathcal{C}$ to $P(V)$ vertically with its bend on $l_h$. Furthermore introduce a vertical grid-line $l_v$ between the vertical grid-lines used by $P(V)$ and the path for the true variable in $\mathcal{C}$. Where $l_h$ and $l_v$ cross, introduce the center of the clause gadget as illustrated on the right side in Figure~\ref{fig:3satrepresentation}. Note that the clause gadget is symmetric in $A$, $B$ and $C$ and hence it can be represented with every center ray pointing into the desired direction.
\end{proof}

Martin Pergel~\cite{Per-11} announced a proof for the NP-completeness of deciding whether a graph has bend-number at most $2$. Still, there is an obvious question.

\begin{quest}
 What is the complexity of recognizing $k$-bend graphs for $k\geq 3$?
\end{quest}

%%%%%%%%%%%%%%%%%%%%%%%%%%%%%%%%%%%%%%%%%%%%%%%%%%%%%%%%%%%%
%%%%%%%%%%%%%%%%%%%%%%%%%%%%%%%%%%%%%%%%%%%%%%%%%%%%%%%%%%%%22
%%%%%%%%%%%%%%%%%%%%%%%%%%%%%%%%%%%%%%%%%%%%%%%%%%%%%%%%%%%%

\section{Open Problems}\label{sec:ope}
Most of this paper considers the extremal question of determining the maximum bend-number of a given graph class. For many graph classes this number remains unknown. The most interesting ones are:
\begin{itemize}
 \item bounded degree graphs, where we know that the bend-number is between $\lceil\frac{\Delta}{2}\rceil$ and $\Delta$ (Question~\ref{quest:Delta}),
 \item claw-free graphs, where we believe that the bend-number is unbounded (Conjecture~\ref{conj:claw}),
 \item graphs with bounded pseudo-arboricity $\pa$, where it is only known that $b(G)\leq 2\pa(G)+1$, see~\cite{Bie-10} (Question~\ref{quest:pa}),
\item graphs of bounded simple treewidth where we conjecture that $b(G)\leq 2\stw(G)-3$ (Conjecture~\ref{conj:stw}).
\item A problem which was not addressed in this paper, but seems interesting to us is determining the maximum bend-number of planar graphs. We know that this value is either $3$ or $4$, see~\cite{Hel-12}.
\end{itemize}

Looking at complete bipartite graphs, we ask for the bend-number of $K_{m,n}$ with fixed $m$ and different values of $n$. In Section~\ref{sec:bip} we give some explicit intermediate values from which lower and upper bounds can be derived. What is the exact behavior of this function? Is there a closed formula?

In Section~\ref{sec:rec} we prove NP-hardness of SINGLE-BEND-RECOGNITION. Hence computing $b(G)$ is NP-hard. But the complexity of recognizing $k$-bend graphs for $k\geq 3$ remains open.

A question of somewhat independent combinatorial flavor turned up in Section~\ref{sec:bip},. How many crossings may two grid paths with an even number of bends have, see Question~\ref{quest:cross}? We determined this value for paths with an odd number of bends, see Lemma~\ref{lem:pnt}.

\section*{Acknowledgments}
We thank Stefan Felsner and Piotr Micek for fruitful discussions and Thomas Hixon for his help with the exposition.

\bibliography{lit}
\bibliographystyle{amsplain}

\end{document}